\documentclass[12pt,reqno]{article}
\usepackage{diagbox}
\usepackage{amsfonts}
\usepackage{bbm}
\usepackage{bm}
\usepackage{pifont}
\usepackage{amsmath}
\usepackage{mathrsfs}
\usepackage{hyperref}
\usepackage{amssymb}
\usepackage{amsmath,color}
\usepackage{amsthm}
\usepackage{amstext}
\usepackage[numbers,sort&compress]{natbib}
\usepackage{amsopn}
\usepackage{graphicx}
\usepackage{multirow}
\usepackage{float}
\usepackage[left=3cm,right=3cm,top=3cm,bottom=3cm]{geometry}
\usepackage[english]{babel}
\allowdisplaybreaks
\usepackage{enumerate}

\newtheorem{theorem}{Theorem}[section]
\newtheorem{proposition}[theorem]{Proposition}
\newtheorem{lemma}[theorem]{Lemma}

\newtheorem{definition}[theorem]{Definition}
\newtheorem{assumption}[theorem]{Assumption}
\newtheorem{remark}[theorem]{Remark}

\newcommand{\frA}{\mathfrak{A}}
\let\Section=\section
\def\section{\setcounter{equation}{0}\Section}

\title{Exponential Stability of the Linear Semigroup and the Global Attractor for a Coupled Suspension Bridge System with Past History\footnote{This work is supported by NSFC (12671232).}}
\author{Jun Zhou\footnote{Email: jzhou@swu.edu.cn}\\\\
		{ School of Mathematics and Statistics, Southwest University},\\
		{ Chongqing 400715, People's Republic of China}}
\date{}

\begin{document}
\maketitle

\begin{abstract}
This paper is concerned with the long-time dynamics of a coupled nonlinear
suspension bridge system with past history, in which the deck and the main
cable interact through the nonlinear restoring force \(k[u-v]^{+}\). Building
on the well-posedness and the solution semigroup established in
\cite{Mukiawa2025}, we first show the exponential stability of the associated
linear semigroup by means of the Gearhart--Pr\"uss theorem together with a
detailed resolvent estimate. We then combine this exponential stability with
the local Lipschitz property of the nonlinearity, the fact that the
nonlinearity acts only on the elastic component, the smoothing of the elastic
component, and the exponential decay of the memory kernels to conclude that
the solution semigroup is quasi-stable in the sense of Chueshov--Lasiecka,
and therefore asymptotically smooth. Together with the dissipativity recalled
from \cite{Mukiawa2025}, this leads to a compact global attractor in the
natural energy space. In particular, our approach may be viewed as closing a
gap in the proof of asymptotic smoothness in \cite{Mukiawa2025}, where a
non-decaying initial-difference term was absorbed into an arbitrarily small
constant without justification.
\end{abstract}



\section{Introduction}\label{sec1}

\noindent\textbf{Background and related works.}
Suspension bridges are among the most important civil engineering structures,
and understanding their structural behavior has attracted considerable
attention from researchers. Early contributions by McKenna and co-authors
\cite{Lazer1990,Glover1989,Lazer1987,McKenna1987,McKenna1990,McKenna1999,McKenna2001}
modeled and investigated the vertical vibrations of a suspension bridge
through the coupled system
\begin{equation}\label{eq:classical}
\begin{cases}
m_1u_{tt}+\mu_1u_{xxxx}+k[u-v]^++\gamma_1u_t=f,\\
m_2v_{tt}-\mu_2v_{xx}-k[u-v]^++\gamma_2v_t=g,
\end{cases}
\end{equation}
where \(u=u(x,t)\) is the vertical displacement of the bridge deck and
\(v=v(x,t)\) is the displacement of the main cable. The positive constants
\(m_1,m_2\) represent the mass per unit length of the bridge deck and cable,
respectively, \(k\) is the spring constant of the suspension cables,
\(\mu_1,\mu_2\) are the flexural rigidity of the deck and the coefficient of
tensile strength of the cable, \(\gamma_1,\gamma_2\) are damping coefficients,
and \(f,g\) are external forces. The nonlinear term \(k[u-v]^+\) with
\([u-v]^+=\max\{u-v,0\}\) describes the restoring force exerted by the
suspenders when the deck is pulled downward relative to the cable; it vanishes
when the deck lies above the cable, reflecting the fact that the suspenders can
only pull and not push. The first equation in \eqref{eq:classical} governs the
vibration of the bridge deck in the vertical plane, while the second equation
describes the vibration of the main cable from which the deck is suspended.
Typical boundary conditions are
\[
u(0,t)=u(\ell,t)=u_{xx}(0,t)=u_{xx}(\ell,t)=0,\qquad
v(0,t)=v(\ell,t)=0,
\]
which correspond to a simply supported beam for the deck and a pinned string
for the cable. For simplicity, in the sequel we set \(m_1=m_2=1\), and we
assume that the external forces are time-independent, i.e.,
\(f_i=f_i(x)\), \(i=1,2\), so that the system is autonomous; this is the
standard setting for the study of global attractors.

A number of authors have established results on well-posedness, asymptotic
stability, and global attractors for \eqref{eq:classical} and its variants; see,
e.g., \cite{Ma2009,Ma2005,An2002,Ahmed1998,Holubova2003,Malik2004,DellOro2015}.
For instance, Ma and Zhong \cite{Ma2009} considered the coupled system
\begin{equation}\label{eq:MaZhong}
\begin{cases}
u_{tt}+\mu_1u_{xxxx}+k[u-v]^++\gamma_1u_t+h_1(u)=f_1,\\
v_{tt}-\mu_2v_{xx}-k[u-v]^++\gamma_2v_t+h_2(v)=f_2,
\end{cases}
\end{equation}
and proved the existence of strong solutions and a global attractor under the
regularity assumptions \(h_1\in C^3(\mathbb{R},\mathbb{R})\),
\(h_2\in C^2(\mathbb{R},\mathbb{R})\), and \(f_1,f_2\in L^2(0,\ell)\).
Similar results were obtained in \cite{Ma2005,An2002} for related steady-state
and evolutionary problems. Dell'Oro et al. \cite{DellOro2015} studied the
asymptotic behavior of coupled linear systems modeling suspension bridges and
obtained decay estimates under suitable damping and structural conditions.
These works reveal that the nonlinear coupling \(k[u-v]^+\) plays a crucial
role in the long-time dynamics and that the analysis of global attractors
requires a careful treatment of the non-smooth nonlinearity.

When the vibration of the main cable is ignored, the main cable holding the
suspenders is replaced by a rigid frame, and the vertical displacement of the
suspension bridge is modeled by the single equation
\begin{equation}\label{eq:single}
u_{tt}+u_{xxxx}+k[u]^++\gamma u_t=f.
\end{equation}
For results related to \eqref{eq:single}, see
\cite{Litcanu2004,Ahmed1998,Bochicchio2010,Kang2016,Park2011,Zhong2007} and the
references therein. In particular, Bochicchio et al. \cite{Bochicchio2010}
investigated the long-term damped dynamics of an extensible suspension bridge,
while Kang \cite{Kang2016} studied the asymptotic behavior of a thermoelastic
suspension bridge equation with linear memory. Park and Kang \cite{Park2011}
proved the existence of global attractors for suspension bridge equations with
nonlinear damping, and Zhong et al. \cite{Zhong2007} established the existence
of strong solutions and global attractors for the suspension bridge equations.
Messaoudi et al. \cite{Messaoudi2016} considered a suspension bridge with
memory and partially hinged boundary conditions and proved the existence of a
global attractor. For recent results on the dynamics of beam or plate systems,
Mukiawa and Enyi \cite{Mukiawa2020,Enyi2020} investigated the well-posedness
and the dynamics for Timoshenko and laminated beam systems with Fourier's law
and infinite memory.
\medskip

\noindent\textbf{Motivation: a point that may deserve further clarification in Mukiawa--Messaoudi.}
In recent years, increasing attention has been paid to suspension bridge models
with memory effects, since the materials of the deck and cable often exhibit
viscoelastic behavior. Following the approach of Dafermos \cite{Dafermos1970},
the memory term can be reformulated by introducing the relative displacement
history
\[
\sigma(x,t,s)=u(x,t)-u(x,t-s),\qquad
\eta(x,t,s)=v(x,t)-v(x,t-s),
\]
which satisfies the transport equations
\[
\sigma_t+\sigma_s-u_t=0,\qquad
\eta_t+\eta_s-v_t=0.
\]
This formulation allows one to treat the infinite history as an additional
state variable and to work in an appropriate Hilbert space. Very recently,
Mukiawa and Messaoudi \cite{Mukiawa2025} studied a coupled nonlinear bridge
system with past history of the form
\[
\begin{cases}
u_{tt}+\gamma_1u_t+\mu_1u_{xxxx}+k[u-v]^+-\int_{-\infty}^{t}g_1(t-s)u_{xxxx}(x,s)\,ds+h_1(u)=f_1(x),\\
v_{tt}+\gamma_2v_t-\mu_2v_{xx}-k[u-v]^++\int_{-\infty}^{t}g_2(t-s)v_{xx}(x,s)\,ds+h_2(v)=f_2(x),
\end{cases}
\]
and established the well-posedness as well as the existence of a compact
global attractor in the natural energy space. It seems, however, that their
proof of asymptotic smoothness may contain a point that deserves further
clarification. More precisely, in
\cite[Proposition~4.1]{Mukiawa2025}, the authors derive the difference estimate
\begin{equation}\label{eq:Mukiawa424}
\|U^1(t)-U^2(t)\|_{\mathcal{H}}^2
\le c\|U_0^1-U_0^2\|_{\mathcal{H}}^2
+ c_{\mathcal{B}}\int_0^t\bigl(\|u^1(s)-u^2(s)\|^2+\|v^1(s)-v^2(s)\|^2\bigr)\,ds,
\end{equation}
where \(c_{\mathcal{B}}>0\) is a constant depending on the absorbing set
\(\mathcal{B}\), and then, in the passage from \eqref{eq:Mukiawa424} to
\cite[(4.25)]{Mukiawa2025}, they assert that for every \(\epsilon>0\) and
sufficiently large \(T>0\),
\begin{equation}\label{eq:Mukiawa425}
\|S(T)U_0^1-S(T)U_0^2\|_{\mathcal{H}}^2
\le \epsilon+\Upsilon_T(U_0^1,U_0^2),
\end{equation}
where \(\Upsilon_T\) is defined as in \cite[(4.26)]{Mukiawa2025}. This step
would appear to require some justification: the first term on the right-hand
side of \eqref{eq:Mukiawa424}, namely
\(c\|U_0^1-U_0^2\|_{\mathcal{H}}^2\), does not decay as \(T\to{+\infty}\),
and hence it is not clear how it could be absorbed into the arbitrarily small
constant \(\epsilon\). Indeed, for two trajectories starting from different
initial data in a bounded absorbing set, the initial difference need not be
small, and no decay factor seems to be available in \eqref{eq:Mukiawa424} to
control it. Moreover, the quantity \(\Upsilon_T\) in \eqref{eq:Mukiawa425}
contains only the integral term involving the differences
\(\|u^1-u^2\|\) and \(\|v^1-v^2\|\) over \([0,T]\); it does not contain the
initial difference \(c\|U_0^1-U_0^2\|_{\mathcal{H}}^2\). Thus the passage from
\eqref{eq:Mukiawa424} to \eqref{eq:Mukiawa425} appears to omit the
non-decaying initial-difference term, and this point would benefit from
further justification. To make this passage rigorous, one needs an exponential
decay estimate for the linear semigroup, which was not established in
\cite{Mukiawa2025}. To the best of our knowledge, this point has not been
addressed in the literature.

Motivated by the above discussion and in order to address the gap in
\cite[Proposition~4.1]{Mukiawa2025}, we consider the following nonlinear coupled
suspension bridge model with infinite history:
\begin{equation}\label{eq:model}
\begin{cases}
u_{tt}\!+\!\gamma_1 u_t\!+\!\mu_1 u_{xxxx}\!+\!k[u-v]^+
\!-\!\displaystyle\int_{-\infty}^{t}g_1(t\!-\!s)u_{xxxx}(x,s)\,ds\!+\!h_1(u)=f_1(x),\\[2mm]
v_{tt}+\gamma_2 v_t-\mu_2 v_{xx}-k[u-v]^+
+\displaystyle\int_{-\infty}^{t}g_2(t-s)v_{xx}(x,s)\,ds+h_2(v)=f_2(x),\\[2mm]
u(0,t)=u_{xx}(0,t)=u(\ell,t)=u_{xx}(\ell,t)=0,\\
v(0,t)=v(\ell,t)=0,\\
u(x,-s)=\varphi_0(x,s),\quad s>0,\\
v(x,-s)=\psi_0(x,s),\quad s>0,\\
u_t(x,0)=u_1(x),\quad v_t(x,0)=v_1(x),
\end{cases}
\end{equation}
where \((x,t)\in(0,\ell)\times\mathbb{R}_+\), \(\ell>0\) is a constant, the
parameters \(\gamma_i,\mu_i\), \(i=1,2\), and \(k\) are positive constants,
\(f_1,f_2\in L^2(0,\ell)\), and \([u-v]^+=\max\{u-v,0\}\). As explained above,
the external forces are assumed to be time-independent, so that the system is
autonomous.

Throughout this work, we impose the following assumptions on the nonlinear
functions \(h_1,h_2\) and the memory kernels \(g_1,g_2\), which are identical to
those adopted in \cite{Mukiawa2025}.

\begin{assumption}\label{assumh}
For each \(i=1,2\), let \(h_i:\mathbb{R}\to\mathbb{R}\) be a \(C^1\) function,
and define its primitive by
\[
H_i(s)=\int_0^s h_i(\tau)\,d\tau.
\]
We assume that there exist constants \(C_1,C_2>0\) such that
\begin{equation}\label{ass1.5}
\liminf_{|s|\to+\infty}\frac{H_i(s)}{s^2}\ge 0,~
\liminf_{|s|\to+\infty}\frac{s h_i(s)-C_1 H_i(s)}{s^2}\ge 0,~
h_i'(s)\ge -C_2,\quad \forall s\in\mathbb{R}.
\end{equation}
Condition \eqref{ass1.5} implies that for any \(\epsilon>0\), there exist
constants \(C_\epsilon,C_\epsilon'>0\) such that
\begin{align}
H_i(s)+\epsilon s^2&\ge -C_\epsilon,\quad \forall s\in\mathbb{R},\; i=1,2,\label{ass1.6}\\
s h_i(s)-C_1 H_i(s)+\epsilon s^2&\ge -C_\epsilon',\quad \forall s\in\mathbb{R},\; i=1,2.\label{ass1.7}
\end{align}
\end{assumption}

\begin{assumption}\label{assumg}
For each \(i=1,2\), the kernel \(g_i:[0,+\infty)\to(0,+\infty)\) is absolutely
continuous and satisfies
\begin{equation}\label{ass1.8}
\int_0^{+\infty} g_i(s)\,ds=\beta_i\in(0,\mu_i),\quad i=1,2,
\end{equation}
and there exists \(\lambda_i>0\), \(i=1,2\), such that for almost every
\(s\in\mathbb{R}^+\),
\begin{equation}\label{ass1.9}
g_i'(s)+\lambda_i g_i(s)\le 0,\quad i=1,2.
\end{equation}
\end{assumption}

The main purpose of this paper is twofold. First, we establish the
exponential stability of the linear semigroup associated with the linearized
system, which provides the compression factor needed in the quasi-stability
framework. Second, using this exponential stability together with the local
Lipschitz property of the nonlinearity, we give a rigorous proof of
quasi-stability in the sense of Chueshov--Lasiecka, and consequently of
asymptotic smoothness and the existence of a compact global attractor.

Compared with the existing literature, the contributions of the present work
may be summarized as follows.

\begin{enumerate}[(i)]
    \item We establish the exponential stability of the linear semigroup
    associated with the coupled suspension bridge system with past history,
    which, as far as we are aware, was not obtained in \cite{Mukiawa2025}.
    This provides the compression factor required in the quasi-stability
    framework of Chueshov--Lasiecka.
    \item We give a rigorous proof of quasi-stability and asymptotic
    smoothness, which, as discussed above, addresses the point in
    \cite[Proposition~4.1]{Mukiawa2025} that seemed to call for further
    justification. Our approach uses the exponential stability of the linear
    semigroup together with the structural property that the nonlinearity acts
    only on the elastic component, so that the nonlinear remainder is
    relatively compact.
    \item Our results also cover the coupled suspension bridge system with
    infinite memory and non-smooth nonlinearity \(k[u-v]^+\), and they require
    less regularity on the nonlinearities \(h_1,h_2\) than the earlier work
    \cite{Ma2009}: we assume \(h_i\in C^1(\mathbb{R},\mathbb{R})\) together
    with the growth and dissipativity conditions \eqref{ass1.5}, whereas
    \cite{Ma2009} requires \(h_1\in C^3(\mathbb{R},\mathbb{R})\) and
    \(h_2\in C^2(\mathbb{R},\mathbb{R})\).
\end{enumerate}

The remainder of this paper is organized as follows. In Section
\ref{sec:prelim}, we collect the functional setting, the well-posedness, and
the definitions and tools from the Chueshov--Lasiecka framework. In Section
\ref{sec4}, we prove the exponential stability of the linear semigroup via the
Gearhart--Pr\"uss theorem. In Section \ref{sec:global}, we prove
quasi-stability, asymptotic smoothness, and the existence of a compact global
attractor.

\section{Preliminaries}\label{sec:prelim}

In this section we collect the functional setting, the well-posedness of system
\eqref{eq:model}, and the definitions and tools from the Chueshov--Lasiecka
framework that will be used in the sequel. The well-posedness was established
in \cite{Mukiawa2025}; we recall it briefly for the reader's convenience.
\medskip

\noindent\textbf{Functional setting and well-posedness.} Let \(\Omega=(0,\ell)\), and let \((\cdot,\cdot)_\Omega\) and
\(\|\cdot\|_\Omega\) denote the inner product and norm in \(L^2(\Omega)\).
Define
\[
H_a^2(\Omega):=\bigl\{\omega\in H^2(\Omega):\omega(0)=\omega(\ell)=0\bigr\},
\]
and the weighted Hilbert spaces
\[
\mathcal{M}_{g_1}:=L^2_{g_1}\bigl(\mathbb{R}^+;H_a^2(\Omega)\bigr),\qquad
\mathcal{M}_{g_2}:=L^2_{g_2}\bigl(\mathbb{R}^+;H_0^1(\Omega)\bigr),
\]
equipped with the inner products
\[
(\sigma,\overline{\sigma})_{\mathcal{M}_{g_1}}
:=\int_0^{+\infty} g_1(s)\bigl(\sigma_{xx}(s),\overline{\sigma}_{xx}(s)\bigr)_\Omega\,ds,
\quad
(\eta,\overline{\eta})_{\mathcal{M}_{g_2}}
:=\int_0^{+\infty} g_2(s)\bigl(\eta_x(s),\overline{\eta}_x(s)\bigr)_\Omega\,ds,
\]
and the corresponding norms. Throughout, \(c_e>0\) and \(c_p>0\) denote the
embedding and Poincar\'e constants satisfying
\begin{equation}\label{eq:embedding}
\|u\|_{L^2(\Omega)}\le\sqrt{c_e}\,\|u_{xx}\|_{L^2(\Omega)},~\,u\in H_a^2(\Omega),
~~
\|v\|_{L^2(\Omega)}\le\sqrt{c_p}\,\|v_x\|_{L^2(\Omega)},~\,v\in H_0^1(\Omega).
\end{equation}
The state space is
\[
\mathcal{H}:=H_a^2(\Omega)\times L^2(\Omega)\times \mathcal{M}_{g_1}
\times H_0^1(\Omega)\times L^2(\Omega)\times\mathcal{M}_{g_2},
\]
which is a Hilbert space equipped with the inner product
\begin{align*}
(U,\overline{U})_{\mathcal{H}}
= \bar{\mu}_1 (u_{xx},\overline{u}_{xx})_\Omega
+ (\psi,\overline{\psi})_\Omega
+ (\sigma,\overline{\sigma})_{\mathcal{M}_{g_1}}\quad + \bar{\mu}_2 (v_x,\overline{v}_x)_\Omega
+ (\phi,\overline{\phi})_\Omega
+ (\eta,\overline{\eta})_{\mathcal{M}_{g_2}}.
\end{align*}

Introduce the relative displacements
\[
\sigma=\sigma(x,t,s):=u(x,t)-u(x,t-s),\qquad
\eta=\eta(x,t,s):=v(x,t)-v(x,t-s),
\]
and set
\begin{equation}\label{barmu}
\bar{\mu}_i=\mu_i-\int_0^{+\infty}g_i(s)\,ds>0,\quad i=1,2.
\end{equation}
Then system \eqref{eq:model} can be rewritten as
\begin{equation}\label{model}
\begin{cases}
u_{tt}+\gamma_1 u_t+\bar{\mu}_1 u_{xxxx}+k[u-v]^+
+\displaystyle\int_0^{+\infty}g_1(s)\sigma_{xxxx}(s)\,ds+h_1(u)=f_1(x),\\[2mm]
\sigma_t+\sigma_s-u_t=0,\\[2mm]
v_{tt}+\gamma_2 v_t-\bar{\mu}_2 v_{xx}-k[u-v]^+
-\displaystyle\int_0^{+\infty}g_2(s)\eta_{xx}(s)\,ds+h_2(v)=f_2(x),\\[2mm]
\eta_t+\eta_s-v_t=0,
\end{cases}
\end{equation}
with initial data
\begin{equation}\label{eq:init}
\begin{cases}
u(x,0)=u_0(x),\quad u_t(x,0)=u_1(x), & x\in(0,\ell),\\[1mm]
\sigma^0(x,s)=u_0(x)-\varphi_0(x,s):=\sigma_0(x,s), & x\in(0,\ell),\; s>0,\\[1mm]
v(x,0)=v_0(x),\quad v_t(x,0)=v_1(x), & x\in(0,\ell),\\[1mm]
\eta^0(x,s)=v_0(x)-\psi_0(x,s):=\eta_0(x,s), & x\in(0,\ell),\; s>0,
\end{cases}
\end{equation}
and boundary conditions
\begin{equation}\label{eq:bound}
\begin{cases}
u(0,t)=u_{xx}(0,t)=u(\ell,t)=u_{xx}(\ell,t)=0, & t\ge 0,\\[1mm]
v(0,t)=v(\ell,t)=0, & t\ge 0,\\[1mm]
\sigma(0,t,s)=\sigma_{xx}(0,t,s)=\sigma(\ell,t,s)=\sigma_{xx}(\ell,t,s)=0, & t\ge 0,\; s>0,\\[1mm]
\eta(0,t,s)=\eta(\ell,t,s)=0, & t\ge 0,\; s>0.
\end{cases}
\end{equation}

For \(U=(u,\psi,\sigma,v,\phi,\eta)\in\mathcal{D}(\mathcal{A})\), the linear
operator \(\mathcal{A}\) and the nonlinearity \(F(U)\) are defined by
\begin{equation}\label{A}
\mathcal{A} U :=
\begin{pmatrix}
-\psi \\[1mm]
\gamma_1 \psi + \bar{\mu}_1 u_{xxxx}
+ \displaystyle\int_0^{+\infty} g_1(s) \sigma_{xxxx}(s) \, ds \\[2mm]
\sigma_s - \psi \\[1mm]
-\phi \\[1mm]
\gamma_2 \phi - \bar{\mu}_2 v_{xx}
- \displaystyle\int_0^{+\infty} g_2(s) \eta_{xx}(s) \, ds \\[2mm]
\eta_s - \phi
\end{pmatrix},
\end{equation}
and
\begin{equation}\label{F}
F(U) :=
\begin{pmatrix}
0 \\[1mm]
-k[u-v]^+ - h_1(u) + f_1 \\[1mm]
0 \\[1mm]
0 \\[1mm]
k[u-v]^+ - h_2(v) + f_2 \\[1mm]
0
\end{pmatrix}.
\end{equation}
The domain of \(\mathcal{A}\) is
\begin{align*}
\mathcal{D}(\mathcal{A})
:= \Big\{ & U=(u,\psi,\sigma,v,\phi,\eta)\in\mathcal{H} \,\Big|\,
\psi\in H_a^2(\Omega),\quad \phi\in H_0^1(\Omega),\\
&\sigma\in\mathcal{D}(\mathcal{M}_{g_1}),\quad
\eta\in\mathcal{D}(\mathcal{M}_{g_2}),\\
&\bar{\mu}_1 u + \int_0^{+\infty} g_1(s)\sigma(s)\,ds
\in H^4(\Omega)\cap H_a^2(\Omega),\\
&\bar{\mu}_2 v + \int_0^{+\infty} g_2(s)\eta(s)\,ds
\in H^2(\Omega)\cap H_0^1(\Omega),\\
&\bar{\mu}_1 u_{xx}(0)
+ \int_0^{+\infty} g_1(s)\sigma_{xx}(0,s)\,ds = 0,\\
&\bar{\mu}_1 u_{xx}(\ell)
+ \int_0^{+\infty} g_1(s)\sigma_{xx}(\ell,s)\,ds = 0 \Big\},
\end{align*}
where
\[
\mathcal{D}(\mathcal{M}_{g_i}):=\bigl\{\sigma\in\mathcal{M}_{g_i} : \sigma_s\in\mathcal{M}_{g_i}\bigr\},\qquad i=1,2.
\]
Then the semigroup formulation of system \eqref{model}--\eqref{eq:init} is
\begin{equation}\label{model-abstract}
\begin{cases}
U_t + \mathcal{A} U = F(U), & t>0,\\
U(0)=U_0:=(u_0,u_1,\sigma_0,v_0,v_1,\eta_0),
\end{cases}
\end{equation}
with \(U=(u,\psi,\sigma,v,\phi,\eta)\), \(\psi:=u_t\), \(\phi:=v_t\).

As shown in \cite[Section~3]{Mukiawa2025}, the operator \(\mathcal{A}\) is
maximal monotone and the nonlinearity \(F\) is locally Lipschitz on
\(\mathcal{H}\). Hence, by the L\"umer--Phillips theorem for the linear part
\(-\mathcal{A}\), which generates a \(C_0\)-semigroup of contractions on
\(\mathcal{H}\), combined with the standard local existence and uniqueness
theory for Lipschitz perturbations of semilinear equations, for every
\(U_0\in\mathcal{H}\) there exists a unique global mild solution
\[
U\in C([0,+\infty);\mathcal{H})
\]
of \eqref{model-abstract}, so the solution semigroup
\[
S(t):\mathcal{H}\to\mathcal{H},\qquad S(t)U_0:=U(t),\quad t\ge0,
\]
is well defined. Since these facts are already established in
\cite{Mukiawa2025}, we shall use them directly in the sequel without further
proof.
\medskip

\noindent\textbf{Global attractors: definition.} For the reader's convenience, we recall the standard definition of a global
attractor.

\begin{definition}[Global attractor]\label{def:attractor}
Let \(X\) be a Banach space and let \(\{S(t)\}_{t\ge0}\) be a continuous
semigroup on \(X\). A set \(\frA\subset X\) is called a \emph{global attractor}
for \(S(t)\) if
\begin{enumerate}[(i)]
    \item \(\frA\) is compact in \(X\);
    \item \(\frA\) is invariant under \(S(t)\), i.e., \(S(t)\frA=\frA\) for all
    \(t\ge0\);
    \item \(\frA\) attracts every bounded subset of \(X\), i.e., for every
    bounded set \(B\subset X\),
    \[
    \lim_{t\to+\infty}d_X\bigl(S(t)B,\frA\bigr)=0,
    \]
    where \(d_X(A,B)=\sup_{a\in A}\inf_{b\in B}\|a-b\|_X\) denotes the
    Hausdorff semi-distance.
\end{enumerate}
\end{definition}

\medskip

\noindent\textbf{Dissipativity and absorbing set.} The energy functional associated with \eqref{model} is
\begin{equation}\label{eq:energy}
\begin{split}
E(t)=&\frac12\Bigl[\|u_t\|^2+\bar{\mu}_1\|u_{xx}\|^2+\|\sigma\|_{\mathcal{M}_{g_1}}^2
+\|v_t\|^2+\bar{\mu}_2\|v_x\|^2+\|\eta\|_{\mathcal{M}_{g_2}}^2\Bigr]\\
&+\frac{k}{2}\|[u-v]^+\|^2+(H_1(u),1)+(H_2(v),1)-(f_1,u)-(f_2,v).
\end{split}
\end{equation}
Under Assumptions \ref{assumh} and \ref{assumg}, it is shown in
\cite[Lemma~4.2]{Mukiawa2025} that \(S(t)\) is dissipative: there exists
\(R_0>0\) such that for every bounded set \(B\subset\mathcal{H}\) there is
\(t_0=t_0(B)\ge0\) with
\begin{equation}\label{eq:absorbing}
S(t)B\subset \mathcal{B}_0:=\{U\in\mathcal{H}:\|U\|_{\mathcal{H}}\le R_0\},
\quad \forall t\ge t_0.
\end{equation}
We shall use \eqref{eq:absorbing} in the sequel and refer to
\cite{Mukiawa2025} for the proof. We also recall that \(F\) is locally
Lipschitz on \(\mathcal{H}\); see \cite[Lemma~3.4]{Mukiawa2025}. More
precisely, for every \(R>0\) there exists \(L_R>0\) such that
\begin{equation}\label{eq:local-Lipschitz}
\|F(U)-F(V)\|_{\mathcal{H}}\le L_R\|U-V\|_{\mathcal{H}},
\qquad
\|U\|_{\mathcal{H}},\|V\|_{\mathcal{H}}\le R.
\end{equation}
Moreover, \(F\) maps bounded sets of \(\mathcal{H}\) into bounded sets of
\(\mathcal{H}\). Finally, on the absorbing ball \(\mathcal{B}_0\) the energy
functional \eqref{eq:energy} is equivalent to the norm of \(\mathcal{H}\) up to
an additive constant; that is, there exist constants \(c_1,c_2>0\) and
\(C_0\ge0\), depending only on the parameters and on \(R_0\), such that
\begin{equation}\label{eq:energy-equivalent}
c_1\|U\|_{\mathcal{H}}^2-C_0\le E(t)\le c_2\|U\|_{\mathcal{H}}^2+C_0,
\qquad U\in\mathcal{B}_0.
\end{equation}
Since all the estimates below are performed on the bounded absorbing ball, the
additive constant \(C_0\) plays no role, and we shall use
\eqref{eq:energy-equivalent} without further mention.

\begin{remark}\label{rem:B0-bounded}
The absorbing ball \(\mathcal{B}_0\) is bounded in \(\mathcal{H}\). Consequently,
there exists \(R_1>R_0\) such that the closed ball
\(\mathcal{B}_1:=\{U\in\mathcal{H}:\|U\|_{\mathcal{H}}\le R_1\}\) contains
\(\mathcal{B}_0\) in its interior. Since \(F\) is locally Lipschitz, there exists
a constant \(L_1>0\), depending only on \(R_1\), such that
\[
\|F(U)-F(V)\|_{\mathcal{H}}\le L_1\|U-V\|_{\mathcal{H}},
\qquad
U,V\in\mathcal{B}_1.
\]
Moreover, \(F(\mathcal{B}_1)\) is bounded in \(\mathcal{H}\). We shall use these
facts repeatedly below.
\end{remark}

\medskip

\noindent\textbf{Quasi-stability and asymptotic smoothness.} We recall the notions of quasi-stability and asymptotic smoothness in the form
that will be used below. We employ the squared-norm formulation of
quasi-stability, which is the standard one in the Chueshov--Lasiecka framework.

\begin{definition}[Quasi-stability]\label{def:qstab}
Let \(\mathcal{H}\) be a Hilbert space, let \(\{S(t)\}_{t\ge0}\) be a continuous
semigroup on \(\mathcal{H}\), and let \(B\subset\mathcal{H}\) be a bounded set.
We say that \(S(t)\) is \emph{quasi-stable} on \(B\) if there exist \(T>0\),
\(q\in[0,1)\), and a relatively compact set \(\mathcal{K}\subset\mathcal{H}\)
such that
\begin{equation}\label{eq:qstab-ineq}
\|S(T)U_0-S(T)V_0\|_{\mathcal{H}}^2
\le q^2\|U_0-V_0\|_{\mathcal{H}}^2
+\|K_T(U_0)-K_T(V_0)\|_{\mathcal{H}}^2,
\qquad\forall U_0,V_0\in B,
\end{equation}
for some map \(K_T:B\to\mathcal{K}\). Here \(\mathcal{K}\) is a relatively
compact subset of \(\mathcal{H}\) containing \(K_T(B)\).
\end{definition}

\begin{definition}[Asymptotic smoothness]\label{def:asym-smooth}
Let \(X\) be a Banach space and let \(\{S(t)\}_{t\ge0}\) be a continuous
semigroup on \(X\). We say that \(S(t)\) is \emph{asymptotically smooth} if
the following two conditions hold:
\begin{enumerate}[(i)]
    \item \(S(t)\) possesses a bounded absorbing set \(B\subset X\), i.e.,
    for every bounded set \(B_0\subset X\) there exists \(t_0=t_0(B_0)\ge0\)
    such that
    \[
    S(t)B_0\subset B,\qquad \forall t\ge t_0;
    \]
    \item for every bounded set \(B_0\subset X\) such that \(S(t)B_0\subset B_0\)
    for all \(t\ge0\), the \(\omega\)-limit set \(\omega(B_0)\) is
    nonempty, compact, and invariant under \(S(t)\), and
    \[
    \lim_{t\to+\infty}d_X\bigl(S(t)B_0,\omega(B_0)\bigr)=0,
    \]
    where \(d_X(A,C)=\sup_{a\in A}\inf_{c\in C}\|a-c\|_X\) denotes the
    Hausdorff semi-distance.
\end{enumerate}
\end{definition}

Equivalently, \(S(t)\) is asymptotically smooth if it is dissipative and every
bounded positively invariant set has a compact \(\omega\)-limit set which
attracts the set. In the framework of Chueshov--Lasiecka, a sufficient
condition for asymptotic smoothness is the existence of a bounded absorbing
set together with the quasi-stability inequality \eqref{eq:qstab-ineq}; this
is precisely the content of the following theorem.

\begin{theorem}[Chueshov--Lasiecka; see {\cite[Theorem~7.9.4]{Chueshov2008}}]
\label{thm:CL-asym}
Let \(\mathcal{H}\) be a Hilbert space and let \(\{S(t)\}_{t\ge0}\) be a
continuous semigroup on \(\mathcal{H}\) possessing a bounded absorbing set
\(B\). Suppose that \(S(t)\) is quasi-stable on \(B\) in the sense of
Definition \ref{def:qstab}. Then \(S(t)\) is asymptotically smooth on
\(\mathcal{H}\) in the sense of Definition \ref{def:asym-smooth}.
\end{theorem}

\section{Exponential Stability}\label{sec4}

In this section we prove the exponential stability of the linear semigroup
\(\{e^{-\mathcal{A}t}\}_{t\ge0}\) generated by \(-\mathcal{A}\). Our strategy
is based on the Gearhart--Pr\"uss theorem, which reduces the proof to the
verification of two conditions on the resolvent of \(\mathcal{A}\). We first
recall the Gearhart--Pr\"uss theorem and establish the monotonicity of
\(\mathcal{A}\), which is needed both for the contraction property of the
linear semigroup and for the description of the approximate point spectrum on
the imaginary axis. Then we verify that the imaginary axis is contained in the
resolvent set of \(\mathcal{A}\), by a contradiction argument based on a Weyl
sequence. Finally, we establish the uniform boundedness of the resolvent on
the imaginary axis by three resolvent estimates. Throughout this section, all
function spaces are understood as complexified Hilbert spaces, and the inner
product is conjugate-linear in the second variable; since all estimates below
are ultimately real-valued, we shall freely take real parts without further
mention.

We first state the main result of this section.

\begin{theorem}[Exponential stability]\label{Exponential stability-precise}
Under Assumption \ref{assumg}, the linear semigroup
\(\{e^{-\mathcal{A}t}\}_{t\ge0}\) generated by \(-\mathcal{A}\), where
\(\mathcal{A}\) is the linear operator defined in \eqref{A}, is exponentially
stable; that is, there exist constants \(\alpha\ge 1\) and \(\beta>0\) such
that
\[
\|e^{-\mathcal{A}t}\|_{\mathcal{L}(\mathcal{H}\to\mathcal{H})}
\le \alpha e^{-\beta t},\quad \forall t\ge0.
\]
\end{theorem}

We then recall the Gearhart--Pr\"uss theorem in the form that will be used
below.

\begin{theorem}[Gearhart--Pr\"uss {\cite{Gearhart1978,Pruss1984}}]\label{GP}
Let $T(t)$ be a bounded $C_0$-semigroup on a Hilbert space $H$ with generator
$-\mathcal{A}$. Then $T(t)$ is exponentially stable if and only if the
following conditions hold:
\begin{enumerate}[(i)]
    \item $i\mathbb{R}\subset\varrho(\mathcal{A})$, where $\varrho(\mathcal{A})$
    is the resolvent set of $\mathcal{A}$;
    \item $\displaystyle\sup_{\lambda\in\mathbb{R}}\|(i\lambda I-\mathcal{A})^{-1}\|_{\mathcal{L}(H)}<{+\infty}$.
\end{enumerate}
\end{theorem}

\begin{remark}\label{rem:GP-equivalent}
In the statement of Theorem \ref{GP}, condition (ii) is the standard uniform
boundedness of the resolvent on the imaginary axis. In the applications below
we shall verify the apparently weaker condition
\[
\limsup_{|\lambda|\to{+\infty}}\|(i\lambda I-\mathcal{A})^{-1}\|_{\mathcal{L}(H)}<{+\infty}.
\]
Once condition (i) holds, the resolvent map
$\lambda\mapsto(i\lambda I-\mathcal{A})^{-1}$ is continuous on
$\mathbb{R}$; hence its boundedness for $|\lambda|$ sufficiently large,
together with its boundedness on every compact interval, implies the uniform
boundedness on the whole real line. Thus the two formulations are equivalent
under condition (i), and we may verify the limit superior version in practice.
In fact, the estimates obtained below yield the stronger uniform bound for all
$\lambda\in\mathbb{R}$, so the limit superior formulation is not needed.
\end{remark}

We first record the monotonicity of $\mathcal{A}$, which was used implicitly in
\cite{Mukiawa2025} and which we shall need below.

\begin{lemma}[Monotonicity of $\mathcal{A}$]\label{lem:monotone}
Under Assumption \ref{assumg}, the operator $\mathcal{A}$ is monotone on
$\mathcal{D}(\mathcal{A})$; that is,
\begin{equation}\label{eq:monotone}
\operatorname{Re}\langle \mathcal{A} U,U\rangle_{\mathcal{H}}\ge 0,
\qquad \forall\,U\in\mathcal{D}(\mathcal{A}).
\end{equation}
More precisely, for every $U=(u,\psi,\sigma,v,\phi,\eta)\in\mathcal{D}(\mathcal{A})$,
\begin{equation}\label{eq:monotone-explicit}
\operatorname{Re}\langle \mathcal{A} U,U\rangle_{\mathcal{H}}
\ge
\gamma_1\|\psi\|^2+\frac{\lambda_1}{2}\|\sigma\|_{\mathcal{M}_{g_1}}^2
+\gamma_2\|\phi\|^2+\frac{\lambda_2}{2}\|\eta\|_{\mathcal{M}_{g_2}}^2.
\end{equation}
Consequently, $-\mathcal{A}$ generates a $C_0$-semigroup of contractions
$\{e^{-\mathcal{A} t}\}_{t\ge0}$ on $\mathcal{H}$.
\end{lemma}

\begin{proof}
Let $U=(u,\psi,\sigma,v,\phi,\eta)\in\mathcal{D}(\mathcal{A})$. By the definition
of $\mathcal{A}$ in \eqref{A} and of the inner product in $\mathcal{H}$, we have
\begin{align*}
\langle \mathcal{A} U,U\rangle_{\mathcal{H}}
&=
\bar{\mu}_1\bigl(-\psi_{xx},u_{xx}\bigr)_\Omega
+\Bigl(\gamma_1\psi+\bar{\mu}_1u_{xxxx}
+\int_0^{+\infty} g_1(s)\sigma_{xxxx}(s)\,ds,\psi\Bigr)_\Omega\\
&\quad
+\bigl(\sigma_s-\psi,\sigma\bigr)_{\mathcal{M}_{g_1}}
+\bar{\mu}_2\bigl(-\phi_x,v_x\bigr)_\Omega
\\
&\quad+\Bigl(\gamma_2\phi-\bar{\mu}_2v_{xx}
-\int_0^{+\infty} g_2(s)\eta_{xx}(s)\,ds,\phi\Bigr)_\Omega
+\bigl(\eta_s-\phi,\eta\bigr)_{\mathcal{M}_{g_2}}.
\end{align*}
Here we used that the first component of $\mathcal{A} U$ is $-\psi$ and the
fourth is $-\phi$, so that the corresponding contributions to the inner product
are $-\bar{\mu}_1(\psi_{xx},u_{xx})_\Omega$ and
$-\bar{\mu}_2(\phi_x,v_x)_\Omega$, respectively.

We now treat the elastic and memory terms separately.

\smallskip
\noindent\textbf{Elastic terms.}
Using the boundary conditions $u=\psi=u_{xx}=\psi_{xx}=0$ at $x=0,\ell$, two
integrations by parts give
\[
\bar{\mu}_1\bigl(-\psi_{xx},u_{xx}\bigr)_\Omega
+\bar{\mu}_1\bigl(u_{xxxx},\psi\bigr)_\Omega
=
-\bar{\mu}_1(u_{xx},\psi_{xx})_\Omega
+\bar{\mu}_1(u_{xx},\psi_{xx})_\Omega.
\]
In the complexified setting the two terms are conjugate to each other, so their
sum is purely imaginary and therefore has zero real part. In particular, the
elastic terms do not contribute to
$\operatorname{Re}\langle\mathcal{A} U,U\rangle_{\mathcal{H}}$. Similarly,
\[
\bar{\mu}_2\bigl(-\phi_x,v_x\bigr)_\Omega
+\bar{\mu}_2\bigl(-v_{xx},\phi\bigr)_\Omega
=
-\bar{\mu}_2(v_x,\phi_x)_\Omega
+\bar{\mu}_2(v_x,\phi_x)_\Omega,
\]
whose sum is again purely imaginary and hence has zero real part.

\smallskip
\noindent\textbf{Memory terms.}
For the first memory term, two integrations by parts yield, using the boundary
conditions $\sigma(0)=\sigma(\ell)=\sigma_{xx}(0)=\sigma_{xx}(\ell)=0$ and
$\psi(0)=\psi(\ell)=\psi_{xx}(0)=\psi_{xx}(\ell)=0$,
\[
\Bigl(\int_0^{+\infty} g_1(s)\sigma_{xxxx}(s)\,ds,\psi\Bigr)_\Omega
=
\Bigl(\int_0^{+\infty} g_1(s)\sigma_{xx}(s)\,ds,\psi_{xx}\Bigr)_\Omega,
\]
with no negative sign. On the other hand, the contribution of the third
component $\sigma_s-\psi$ to the inner product contains the term
\[
\bigl(-\psi,\sigma\bigr)_{\mathcal{M}_{g_1}}
=
-\int_0^{+\infty} g_1(s)\bigl(\psi_{xx},\sigma_{xx}(s)\bigr)_\Omega\,ds.
\]
Therefore, the sum of the two off-diagonal memory terms is
\[
Z_1:=\int_0^{+\infty} g_1(s)\bigl(\sigma_{xx}(s),\psi_{xx}\bigr)_\Omega\,ds
-\int_0^{+\infty} g_1(s)\bigl(\psi_{xx},\sigma_{xx}(s)\bigr)_\Omega\,ds.
\]
Since the inner product is conjugate-linear in the second variable, we have
\[
(\psi_{xx},\sigma_{xx}(s))_\Omega
=
\overline{(\sigma_{xx}(s),\psi_{xx})_\Omega},
\]
and hence
\begin{align*}
Z_1
&=
\int_0^{+\infty} g_1(s)\bigl(\sigma_{xx}(s),\psi_{xx}\bigr)_\Omega\,ds
-\overline{\int_0^{+\infty} g_1(s)\bigl(\sigma_{xx}(s),\psi_{xx}\bigr)_\Omega\,ds}\\
&=
2i\operatorname{Im}\int_0^{+\infty} g_1(s)\bigl(\sigma_{xx}(s),\psi_{xx}\bigr)_\Omega\,ds,
\end{align*}
which is purely imaginary and therefore does not contribute to the real part.
The same argument applies to the $\eta$-terms. Consequently, taking the real
part of $\langle \mathcal{A} U,U\rangle_{\mathcal{H}}$ gives
\begin{align*}
\operatorname{Re}\langle \mathcal{A} U,U\rangle_{\mathcal{H}}
&=
\gamma_1\|\psi\|^2+\gamma_2\|\phi\|^2
+\operatorname{Re}\bigl(\sigma_s,\sigma\bigr)_{\mathcal{M}_{g_1}}
+\operatorname{Re}\bigl(\eta_s,\eta\bigr)_{\mathcal{M}_{g_2}}.
\end{align*}
Finally, we treat the memory terms. By Assumption \ref{assumg}, we have
$g_i(s)\le g_i(0)e^{-\lambda_i s}$ for almost every $s\ge0$, so that
$g_i({+\infty})=0$. Moreover, $\sigma(x,t,0)=u(x,t)-u(x,t)=0$ and
$\eta(x,t,0)=v(x,t)-v(x,t)=0$. Integrating by parts in $s$ and using these
facts, we obtain
\[
\bigl(\sigma_s,\sigma\bigr)_{\mathcal{M}_{g_1}}
=
\frac12\int_0^{+\infty} g_1(s)\frac{d}{ds}\|\sigma_{xx}(s)\|^2\,ds
=
-\frac12\int_0^{+\infty} g_1'(s)\|\sigma_{xx}(s)\|^2\,ds,
\]
and the conclusion follows from $g_1'+\lambda_1g_1\le0$. The same argument
applies to $\eta$. Therefore \eqref{eq:monotone-explicit} holds, and since
$\gamma_i>0$ and $\lambda_i>0$, this implies \eqref{eq:monotone}. Finally,
since $\mathcal{A}$ is maximal monotone by \cite[Section~3]{Mukiawa2025}, the
L\"umer--Phillips theorem shows that $-\mathcal{A}$ generates a $C_0$-semigroup of
contractions.
\end{proof}

\begin{proof}[Proof of Theorem \ref{Exponential stability-precise}]
By Lemma \ref{lem:monotone}, $-\mathcal{A}$ generates a $C_0$-semigroup of
contractions, so the Gearhart--Pr\"uss theorem (Theorem \ref{GP}) applies.

We first prove condition (i): $i\mathbb{R}\subset\varrho(\mathcal{A})$. Assume
for contradiction that there exists $\omega\in\mathbb{R}$ such that
$i\omega\notin\varrho(\mathcal{A})$. Since $\mathcal{A}$ is maximal monotone,
it is a closed operator, and for every $\lambda>0$ the operator
$\lambda I+\mathcal{A}$ is bijective. Since $-\mathcal{A}$ generates a
contraction semigroup, every point of $i\mathbb{R}\cap\sigma(\mathcal{A})$ is an
approximate eigenvalue; see, e.g., \cite[Chapter~2]{Pazy1983}. Hence
$i\omega\notin\varrho(\mathcal{A})$ implies
$i\omega\in\sigma_{\mathrm{ap}}(\mathcal{A})$. Therefore there exists a
sequence
$\{U_n\}=(u_n,\psi_n,\sigma_n,v_n,\phi_n,\eta_n)\in\mathcal{D}(\mathcal{A})$
with $\|U_n\|_{\mathcal{H}}=1$ for all $n\in\mathbb{N}$, satisfying
\begin{equation}\label{eq:resolvent_approx}
F_n:=i\omega U_n-\mathcal{A}U_n\to0\quad\text{in }\mathcal{H}.
\end{equation}
Writing $F_n=(f_n^1,f_n^2,f_n^3,f_n^4,f_n^5,f_n^6)$, the componentwise
convergence in \eqref{eq:resolvent_approx} reads
\begin{equation}\label{eq:component_convergence}
\begin{cases}
f_n^1=i\omega u_n+\psi_n\to0 & \text{in }H_a^2(\Omega),\\
f_n^2=i\omega\psi_n-\gamma_1\psi_n-\bar{\mu}_1 u_{n,xxxx}
-\displaystyle\int_0^{+\infty}g_1(s)\sigma_{n,xxxx}(s)\,ds\to0
& \text{in }L^2(\Omega),\\
f_n^3=i\omega\sigma_n-\sigma_{n,s}+\psi_n\to0 & \text{in }\mathcal{M}_{g_1},\\
f_n^4=i\omega v_n+\phi_n\to0 & \text{in }H_0^1(\Omega),\\
f_n^5=i\omega\phi_n-\gamma_2\phi_n+\bar{\mu}_2 v_{n,xx}
+\displaystyle\int_0^{+\infty}g_2(s)\eta_{n,xx}(s)\,ds\to0
& \text{in }L^2(\Omega),\\
f_n^6=i\omega\eta_n-\eta_{n,s}+\phi_n\to0 & \text{in }\mathcal{M}_{g_2}.
\end{cases}
\end{equation}
From \eqref{eq:resolvent_approx}, Lemma \ref{lem:monotone}, and the
boundedness of $U_n$ in $\mathcal{H}$, we deduce
\begin{align}\label{eq:pq}
\gamma_1\|\psi_n\|^2+\frac{\lambda_1}{2}\|\sigma_n\|_{\mathcal{M}_{g_1}}^2
+\gamma_2\|\phi_n\|^2+\frac{\lambda_2}{2}\|\eta_n\|_{\mathcal{M}_{g_2}}^2
&\le \operatorname{Re}\langle \mathcal{A} U_n,U_n\rangle_{\mathcal{H}}\notag\\
&= -\operatorname{Re}\langle i\omega U_n-\mathcal{A}U_n,U_n\rangle_{\mathcal{H}}\notag\\
&\le \|F_n\|_{\mathcal{H}}\|U_n\|_{\mathcal{H}}\to0.
\end{align}
Here we used that $\operatorname{Re}\langle i\omega U_n,U_n\rangle_{\mathcal{H}}=0$.
Consequently,
\begin{equation}\label{gj:pq}
\|\psi_n\|^2,\ \|\sigma_n\|_{\mathcal{M}_{g_1}}^2,\ \|\phi_n\|^2,\ \|\eta_n\|_{\mathcal{M}_{g_2}}^2\to0.
\end{equation}
Since $\psi_n\to0$ and $\phi_n\to0$ in $L^2(\Omega)$ by \eqref{gj:pq}, the
damping terms $\gamma_1\psi_n$ and $\gamma_2\phi_n$ can be dropped in the
limit, and equation \eqref{eq:resolvent_approx} simplifies to
\begin{equation}\label{eq:convergence_1}
\begin{cases}
f_n^1=i\omega u_n+\psi_n\to0 & \text{in }H_a^2(\Omega),\\
f_n^2=i\omega\psi_n-\bar{\mu}_1 u_{n,xxxx}
-\displaystyle\int_0^{+\infty}g_1(s)\sigma_{n,xxxx}(s)\,ds\to0
& \text{in }L^2(\Omega),\\
f_n^3=i\omega\sigma_n-\sigma_{n,s}+\psi_n\to0 & \text{in }\mathcal{M}_{g_1},\\
f_n^4=i\omega v_n+\phi_n\to0 & \text{in }H_0^1(\Omega),\\
f_n^5=i\omega\phi_n+\bar{\mu}_2 v_{n,xx}
+\displaystyle\int_0^{+\infty}g_2(s)\eta_{n,xx}(s)\,ds\to0
& \text{in }L^2(\Omega),\\
f_n^6=i\omega\eta_n-\eta_{n,s}+\phi_n\to0 & \text{in }\mathcal{M}_{g_2}.
\end{cases}
\end{equation}
We now split the proof into two cases.

\medskip
\noindent\textbf{Case 1: $\omega=0$.}
From \eqref{eq:convergence_1}$_1$ and \eqref{eq:convergence_1}$_4$, we obtain
\begin{align}\label{ssf_1}
\psi_n\to0\quad\text{in }H_a^2(\Omega),\qquad
\phi_n\to0\quad\text{in }H_0^1(\Omega).
\end{align}
By Poincar\'e's inequality,
\begin{align}\label{ssf_2}
\psi_n\to0\quad\text{in }L^2(\Omega),\qquad
\phi_n\to0\quad\text{in }L^2(\Omega).
\end{align}
From \eqref{eq:convergence_1}$_3$ and \eqref{eq:convergence_1}$_6$, together
with \eqref{ssf_2}, we derive
\begin{align}\label{ssf_3}
\sigma_{n,s}\to0\quad\text{in }\mathcal{M}_{g_1},\qquad
\eta_{n,s}\to0\quad\text{in }\mathcal{M}_{g_2}.
\end{align}
Taking the $L^2$-inner products of \eqref{eq:convergence_1}$_2$ with $u_n$ and
\eqref{eq:convergence_1}$_5$ with $v_n$, respectively, and then taking real
parts, yields
\begin{align*}
\bar{\mu}_1\|u_{n,xx}\|^2+\bar{\mu}_2\|v_{n,x}\|^2
&=-\operatorname{Re}(f_n^2,u_n)-\operatorname{Re}(f_n^5,v_n)\\
&\quad-\operatorname{Re}\left(\int_0^{+\infty} g_1(s)\sigma_{n,xx}(s)\,ds,u_{n,xx}\right)\\
&\quad-\operatorname{Re}\left(\int_0^{+\infty} g_2(s)\eta_{n,x}(s)\,ds,v_{n,x}\right).
\end{align*}
Since $\bar{\mu}_1\|u_{n,xx}\|^2+\bar{\mu}_2\|v_{n,x}\|^2
\le\|U_n\|_{\mathcal{H}}=1$, by Poincar\'e's inequality and H\"older's
inequality, there exists a positive constant $c$ such that
\begin{align*}
\bar{\mu}_1\|u_{n,xx}\|^2+\bar{\mu}_2\|v_{n,x}\|^2
&\le c\|f_n^2\|\|u_n\|+c\|f_n^5\|\|v_n\|\\
&\quad+\|u_{n,xx}\|\int_0^{+\infty} g_1(s)\|\sigma_{n,xx}(s)\|\,ds
+\|v_{n,x}\|\int_0^{+\infty} g_2(s)\|\eta_{n,x}(s)\|\,ds\\
&\le c\|f_n^2\|+c\|f_n^5\|\\
&\quad+\left(\int_0^{+\infty} g_1(s)\,ds\right)^{1/2}\|\sigma_n\|_{\mathcal{M}_{g_1}}
+\left(\int_0^{+\infty} g_2(s)\,ds\right)^{1/2}\|\eta_n\|_{\mathcal{M}_{g_2}}.
\end{align*}
From \eqref{eq:convergence_1}$_2$, \eqref{eq:convergence_1}$_5$, and
\eqref{gj:pq}, we conclude
\begin{align}\label{ssf_4}
u_{n,xx}\to0,\quad v_{n,x}\to0\quad\text{in }L^2(\Omega).
\end{align}
Hence, from \eqref{gj:pq} and \eqref{ssf_4}, we obtain $\|U_n\|\to0$ in
$\mathcal{H}$, which contradicts $\|U_n\|_{\mathcal{H}}=1$ for all
$n\in\mathbb{N}$.

\medskip
\noindent\textbf{Case 2: $\omega\ne0$.}
From \eqref{eq:convergence_1}$_1$ and \eqref{gj:pq}, the relation
$i\omega u_n=f_n^1-\psi_n$ gives
\begin{equation}\label{eq:u_n-small-L2}
u_n\to0\quad\text{in }L^2(\Omega).
\end{equation}
Similarly, from \eqref{eq:convergence_1}$_4$ and \eqref{gj:pq}, the relation
$i\omega v_n=f_n^4-\phi_n$ gives
\begin{equation}\label{eq:v_n-small-L2}
v_n\to0\quad\text{in }L^2(\Omega).
\end{equation}
We prove the two remaining convergences in two steps.

\smallskip
\noindent\textbf{Step 1: $\|u_{n,xx}\|\to0$.}
Taking the $L^2$-inner products of \eqref{eq:convergence_1}$_2$ with $u_n$ and
\eqref{eq:convergence_1}$_1$ with $\psi_n$, respectively, and then taking real
parts, we obtain
\begin{align}\label{eq:case2-u-identity}
 \bar{\mu}_1\|u_{n,xx}\|^2
=
&-\operatorname{Re}(f_n^2,u_n)-\operatorname{Re}(f_n^1,\psi_n)
\notag\\
&+\|\psi_n\|^2
-\operatorname{Re}\int_0^{+\infty} g_1(s)\bigl(\sigma_{n,xx}(s),u_{n,xx}\bigr)_\Omega\,ds.
\end{align}
Indeed, from \eqref{eq:convergence_1}$_2$,
\[
(f_n^2,u_n)
=
i\omega(\psi_n,u_n)-\bar{\mu}_1(u_{n,xxxx},u_n)
-\int_0^{+\infty} g_1(s)(\sigma_{n,xxxx}(s),u_n)_\Omega\,ds,
\]
and two integrations by parts give
\[
(u_{n,xxxx},u_n)=\|u_{n,xx}\|^2,\qquad
(\sigma_{n,xxxx}(s),u_n)_\Omega=(\sigma_{n,xx}(s),u_{n,xx})_\Omega.
\]
Hence
\[
\operatorname{Re}(f_n^2,u_n)
=
\operatorname{Re}\bigl[i\omega(\psi_n,u_n)\bigr]
-\bar{\mu}_1\|u_{n,xx}\|^2
-\operatorname{Re}\int_0^{+\infty} g_1(s)(\sigma_{n,xx}(s),u_{n,xx})_\Omega\,ds.
\]
On the other hand, from \eqref{eq:convergence_1}$_1$,
\[
(f_n^1,\psi_n)=i\omega(u_n,\psi_n)+\|\psi_n\|^2,
\]
so that
\[
\operatorname{Re}(f_n^1,\psi_n)
=
\operatorname{Re}\bigl[i\omega(u_n,\psi_n)\bigr]+\|\psi_n\|^2.
\]
Since the inner product is conjugate-linear in the second variable,
\[
\operatorname{Re}\bigl[i\omega(\psi_n,u_n)\bigr]
+
\operatorname{Re}\bigl[i\omega(u_n,\psi_n)\bigr]=0.
\]
Adding the two real parts therefore yields \eqref{eq:case2-u-identity}.
Now, by Poincar\'e's inequality, H\"older's inequality, and
$\|U_n\|_{\mathcal{H}}=1$, there exists a positive constant $c$ such that
\begin{align*}
\bar{\mu}_1\|u_{n,xx}\|^2
&\le c\|f_n^2\|\|u_n\|+c\|f_n^1\|\|\psi_n\|+\|\psi_n\|^2\\
&\quad+\|u_{n,xx}\|\int_0^{+\infty} g_1(s)\|\sigma_{n,xx}(s)\|\,ds\\
&\le c\|f_n^2\|+c\|f_n^1\|+\|\psi_n\|^2
+\left(\int_0^{+\infty} g_1(s)\,ds\right)^{1/2}\|\sigma_n\|_{\mathcal{M}_{g_1}},
\end{align*}
where we used $\|u_n\|\le\|U_n\|_{\mathcal{H}}=1$. From
\eqref{eq:convergence_1}$_1$, \eqref{eq:convergence_1}$_2$,
\eqref{gj:pq}, and \eqref{eq:u_n-small-L2}, we conclude
\begin{align}\label{ssf_5}
u_{n,xx}\to0\quad\text{in }L^2(\Omega).
\end{align}

\smallskip
\noindent\textbf{Step 2: $\|v_{n,x}\|\to0$.}
Taking the $L^2$-inner products of \eqref{eq:convergence_1}$_5$ with $v_n$ and
\eqref{eq:convergence_1}$_4$ with $\phi_n$, respectively, and then taking real
parts, we obtain
\begin{equation}\label{eq:case2-v-identity}
\bar{\mu}_2\|v_{n,x}\|^2
=
-\operatorname{Re}(f_n^5,v_n)-\operatorname{Re}(f_n^4,\phi_n)
+\|\phi_n\|^2
-\operatorname{Re}\int_0^{+\infty} g_2(s)\bigl(\eta_{n,x}(s),v_{n,x}\bigr)_\Omega\,ds.
\end{equation}
Indeed, from \eqref{eq:convergence_1}$_5$,
\[
(f_n^5,v_n)
=
i\omega(\phi_n,v_n)+\bar{\mu}_2(v_{n,xx},v_n)
+\int_0^{+\infty} g_2(s)(\eta_{n,xx}(s),v_n)_\Omega\,ds,
\]
and two integrations by parts give
\[
(v_{n,xx},v_n)=-\|v_{n,x}\|^2,\qquad
(\eta_{n,xx}(s),v_n)_\Omega=-(\eta_{n,x}(s),v_{n,x})_\Omega.
\]
Hence
\[
\operatorname{Re}(f_n^5,v_n)
=
\operatorname{Re}\bigl[i\omega(\phi_n,v_n)\bigr]
-\bar{\mu}_2\|v_{n,x}\|^2
-\operatorname{Re}\int_0^{+\infty} g_2(s)(\eta_{n,x}(s),v_{n,x})_\Omega\,ds.
\]
On the other hand, from \eqref{eq:convergence_1}$_4$,
\[
(f_n^4,\phi_n)=i\omega(v_n,\phi_n)+\|\phi_n\|^2,
\]
so that
\[
\operatorname{Re}(f_n^4,\phi_n)
=
\operatorname{Re}\bigl[i\omega(v_n,\phi_n)\bigr]+\|\phi_n\|^2.
\]
Since the inner product is conjugate-linear in the second variable,
\[
\operatorname{Re}\bigl[i\omega(\phi_n,v_n)\bigr]
+
\operatorname{Re}\bigl[i\omega(v_n,\phi_n)\bigr]=0.
\]
Adding the two real parts therefore yields \eqref{eq:case2-v-identity}.
By Poincar\'e's inequality, H\"older's inequality, and
$\|U_n\|_{\mathcal{H}}=1$, there exists a positive constant $c$ such that
\begin{align*}
\bar{\mu}_2\|v_{n,x}\|^2
&\le c\|f_n^5\|\|v_n\|+c\|f_n^4\|\|\phi_n\|+\|\phi_n\|^2+\|v_{n,x}\|\int_0^{+\infty} g_2(s)\|\eta_{n,x}(s)\|\,ds\\
&\le c\|f_n^5\|+c\|f_n^4\|+\|\phi_n\|^2
+\left(\int_0^{+\infty} g_2(s)\,ds\right)^{1/2}\|\eta_n\|_{\mathcal{M}_{g_2}},
\end{align*}
where we used $\|v_n\|\le\|U_n\|_{\mathcal{H}}=1$. From
\eqref{eq:convergence_1}$_4$, \eqref{eq:convergence_1}$_5$,
\eqref{gj:pq}, and \eqref{eq:v_n-small-L2}, we conclude
\begin{align}\label{ssf_6}
v_{n,x}\to0\quad\text{in }L^2(\Omega).
\end{align}
Hence, from \eqref{gj:pq}, \eqref{ssf_5}, and \eqref{ssf_6}, we obtain
$\|U_n\|\to0$ in $\mathcal{H}$, which contradicts $\|U_n\|_{\mathcal{H}}=1$
for all $n\in\mathbb{N}$.

Next, we prove condition (ii), namely the uniform boundedness of the resolvent
operator $(i\lambda I-\mathcal{A})^{-1}$ for $|\lambda|$ sufficiently large.
Let $\Upsilon=(h^1,\ldots,h^6)\in\mathcal{H}$ be given, and let
$U=(u,\psi,\sigma,v,\phi,\eta)\in\mathcal{D}(\mathcal{A})$ satisfy the
resolvent equation
\begin{align}\label{eq:rq}
i\lambda U-\mathcal{A}U=\Upsilon,
\end{align}
which in components reads
\begin{equation}\label{eq:boundedness}
\begin{cases}
i\lambda u+\psi=h^1,\\
i\lambda\psi-\gamma_1\psi-\bar{\mu}_1 u_{xxxx}
-\displaystyle\int_0^{+\infty}g_1(s)\sigma_{xxxx}(s)\,ds=h^2,\\
i\lambda\sigma-\sigma_s+\psi=h^3,\\
i\lambda v+\phi=h^4,\\
i\lambda\phi-\gamma_2\phi+\bar{\mu}_2 v_{xx}
+\displaystyle\int_0^{+\infty}g_2(s)\eta_{xx}(s)\,ds=h^5,\\
i\lambda\eta-\eta_s+\phi=h^6.
\end{cases}
\end{equation}
Our objective is to establish the existence of a constant $C>0$, independent
of $\lambda$, such that for every $\lambda\in\mathbb{R}$ and every
$\Upsilon\in\mathcal{H}$, the solution $U$ of \eqref{eq:rq} satisfies
\begin{align}\label{bound}
\|U\|_{\mathcal{H}}\le C\|\Upsilon\|_{\mathcal{H}}.
\end{align}
Once \eqref{bound} is proved for all $\lambda\in\mathbb{R}$, condition (ii) of
Theorem \ref{GP} follows immediately. In fact, the estimates below do not
require $|\lambda|$ to be large; they hold uniformly in $\lambda$.

\begin{lemma}\label{eq:hsx}
Under Assumption \ref{assumg}, there exists a constant $C>0$, independent of
$\lambda$, such that
\begin{align}\label{eq:hsx_estimate}
\|\psi\|^2,\ \|\sigma\|_{\mathcal{M}_{g_1}}^2,\ \|\phi\|^2,\ \|\eta\|_{\mathcal{M}_{g_2}}^2
\le C\|U\|_{\mathcal{H}}\|\Upsilon\|_{\mathcal{H}}.
\end{align}
\end{lemma}

\begin{proof}
By Lemma \ref{lem:monotone}, combined with \eqref{eq:rq}, we deduce
\begin{align*}
\gamma_1\|\psi\|^2+\frac{\lambda_1}{2}\|\sigma\|_{\mathcal{M}_{g_1}}^2
+\gamma_2\|\phi\|^2+\frac{\lambda_2}{2}\|\eta\|_{\mathcal{M}_{g_2}}^2
&\le \operatorname{Re}\langle \mathcal{A} U,U\rangle_{\mathcal{H}}\\
&=-\operatorname{Re}\langle i\lambda U-\mathcal{A}U,U\rangle_{\mathcal{H}}\\
&=-\operatorname{Re}\langle\Upsilon,U\rangle_{\mathcal{H}}
\le \|U\|_{\mathcal{H}}\|\Upsilon\|_{\mathcal{H}}.
\end{align*}
Thus \eqref{eq:hsx_estimate} follows.
\end{proof}

\begin{lemma}\label{eq:u}
Under Assumption \ref{assumg}, there exists a constant $C>0$, independent of
$\lambda$, such that
\begin{align}\label{eq:u_estimate}
\bar{\mu}_1\|u_{xx}\|^2\le C\|U\|_{\mathcal{H}}\|\Upsilon\|_{\mathcal{H}}.
\end{align}
\end{lemma}

\begin{proof}
From \eqref{eq:boundedness}$_1$ we have $i\lambda u=h^1-\psi$. Taking the
$L^2$-inner product of \eqref{eq:boundedness}$_2$ with $u$ and using this
relation to replace $i\lambda u$, we obtain
\begin{align*}
(h^2,u)
&=
(i\lambda\psi-\gamma_1\psi-\bar{\mu}_1 u_{xxxx}
-\int_0^{+\infty} g_1(s)\sigma_{xxxx}(s)\,ds,u)\\
&=
(\psi,\overline{i\lambda u})-\gamma_1(\psi,u)-\bar{\mu}_1\|u_{xx}\|^2
-\int_0^{+\infty} g_1(s)(\sigma_{xx}(s),u_{xx})_\Omega\,ds\\
&=
(\psi,\overline{h^1})-\|\psi\|^2-\gamma_1(\psi,u)-\bar{\mu}_1\|u_{xx}\|^2
-\int_0^{+\infty} g_1(s)(\sigma_{xx}(s),u_{xx})_\Omega\,ds.
\end{align*}
Here we used $\overline{i\lambda u}=\overline{h^1}-\overline{\psi}$, so that
$(\psi,\overline{i\lambda u})=(\psi,\overline{h^1})-\|\psi\|^2$. Taking real
parts gives
\begin{align*}
\bar{\mu}_1\|u_{xx}\|^2
&=
\operatorname{Re}(\psi,\overline{h^1})-\|\psi\|^2-\gamma_1\operatorname{Re}(\psi,u)
-\operatorname{Re}(h^2,u)\\
&\quad-\operatorname{Re}\int_0^{+\infty} g_1(s)(\sigma_{xx}(s),u_{xx})_\Omega\,ds.
\end{align*}
By the definition of the norm in $\mathcal{H}$,
\[
\|U\|_{\mathcal{H}}\le C\bigl(\|u_{xx}\|+\|\psi\|+\|\sigma\|_{\mathcal{M}_{g_1}}
+\|v_x\|+\|\phi\|+\|\eta\|_{\mathcal{M}_{g_2}}\bigr).
\]
Then, by Poincar\'e's inequality, Young's inequality, and H\"older's
inequality, there exists a positive constant $C$ such that
\begin{align*}
\bar{\mu}_1\|u_{xx}\|^2
&\le C\|h^1\|\|\psi\|+\|\psi\|^2+\gamma_1\|\psi\|\|u\|
+C\|h^2\|\|u\|+\|u_{xx}\|\int_0^{+\infty} g_1(s)\|\sigma_{xx}(s)\|\,ds\\
&\le C\|\Upsilon\|_{\mathcal{H}}\|U\|_{\mathcal{H}}
+\|\psi\|^2+\frac{\gamma_1^2}{4\epsilon}\|\psi\|^2+\epsilon\|u\|^2
+\epsilon_1\|u_{xx}\|^2\\
&\quad+\frac{1}{4\epsilon_1}\left(\int_0^{+\infty} g_1(s)\,ds\right)\|\sigma\|_{\mathcal{M}_{g_1}}^2,
\end{align*}
where we used $\|h^1\|\le\|\Upsilon\|_{\mathcal{H}}$ and
$\|h^2\|\le\|\Upsilon\|_{\mathcal{H}}$. Using \eqref{eq:hsx_estimate} to
control the terms involving $\|\psi\|^2$, $\|\sigma\|_{\mathcal{M}_{g_1}}^2$,
$\|\phi\|^2$ and $\|\eta\|_{\mathcal{M}_{g_2}}^2$, and choosing $\epsilon>0$
and $\epsilon_1>0$ sufficiently small so that
$\epsilon c_e+\epsilon_1<\bar{\mu}_1/2$, where $c_e$ is the embedding
constant in \eqref{eq:embedding}, we can absorb the terms involving
$\|u_{xx}\|^2$ and $\|u\|^2$ into the left-hand side. We obtain
\[
\bar{\mu}_1\|u_{xx}\|^2
\le C\|U\|_{\mathcal{H}}\|\Upsilon\|_{\mathcal{H}}
+\frac12\bar{\mu}_1\|u_{xx}\|^2,
\]
hence \eqref{eq:u_estimate}. The parameters $\epsilon,\epsilon_1$ are chosen
independently of $\lambda$, and so is the constant $C$.
\end{proof}

\begin{lemma}\label{eq:v}
Under Assumption \ref{assumg}, there exists a constant $C>0$, independent of
$\lambda$, such that
\begin{align}\label{eq:v_estimate}
\bar{\mu}_2\|v_x\|^2\le C\|U\|_{\mathcal{H}}\|\Upsilon\|_{\mathcal{H}}.
\end{align}
\end{lemma}

\begin{proof}
From \eqref{eq:boundedness}$_4$ we have $i\lambda v=h^4-\phi$. Taking the
$L^2$-inner product of \eqref{eq:boundedness}$_5$ with $v$ and using this
relation to replace $i\lambda v$, we obtain
\begin{align*}
(h^5,v)
&=
(i\lambda\phi-\gamma_2\phi+\bar{\mu}_2 v_{xx}
+\int_0^{+\infty} g_2(s)\eta_{xx}(s)\,ds,v)\\
&=
(\phi,\overline{i\lambda v})-\gamma_2(\phi,v)+\bar{\mu}_2(v_{xx},v)
+\int_0^{+\infty} g_2(s)(\eta_{xx}(s),v)_\Omega\,ds\\
&=
(\phi,\overline{h^4})-\|\phi\|^2-\gamma_2(\phi,v)-\bar{\mu}_2\|v_x\|^2
-\int_0^{+\infty} g_2(s)(\eta_x(s),v_x)_\Omega\,ds.
\end{align*}
Here we used $\overline{i\lambda v}=\overline{h^4}-\overline{\phi}$, so that
$(\phi,\overline{i\lambda v})=(\phi,\overline{h^4})-\|\phi\|^2$, together with
$(v_{xx},v)=-\|v_x\|^2$ and
$(\eta_{xx}(s),v)_\Omega=-(\eta_x(s),v_x)_\Omega$. Taking real parts gives
\begin{align*}
\bar{\mu}_2\|v_x\|^2
=
\operatorname{Re}(\phi,\overline{h^4})-\|\phi\|^2-\gamma_2\operatorname{Re}(\phi,v)
-\operatorname{Re}(h^5,v)-\operatorname{Re}\int_0^{+\infty} g_2(s)(\eta_x(s),v_x)_\Omega\,ds.
\end{align*}
By the definition of the norm in $\mathcal{H}$,
\[
\|U\|_{\mathcal{H}}\le C\bigl(\|u_{xx}\|+\|\psi\|+\|\sigma\|_{\mathcal{M}_{g_1}}
+\|v_x\|+\|\phi\|+\|\eta\|_{\mathcal{M}_{g_2}}\bigr).
\]
Then, by Poincar\'e's inequality, Young's inequality, and H\"older's
inequality, there exists a positive constant $C$ such that
\begin{align*}
\bar{\mu}_2\|v_x\|^2
&\le C\|h^4\|\|\phi\|+\|\phi\|^2+\gamma_2\|\phi\|\|v\|
+C\|h^5\|\|v\|+\|v_x\|\int_0^{+\infty} g_2(s)\|\eta_x(s)\|\,ds\\
&\le C\|\Upsilon\|_{\mathcal{H}}\|U\|_{\mathcal{H}}
+\|\phi\|^2+\frac{\gamma_2^2}{4\delta}\|\phi\|^2+\delta\|v\|^2
+\delta_1\|v_x\|^2\\
&\quad+\frac{1}{4\delta_1}\left(\int_0^{+\infty} g_2(s)\,ds\right)\|\eta\|_{\mathcal{M}_{g_2}}^2,
\end{align*}
where we used $\|h^4\|\le\|\Upsilon\|_{\mathcal{H}}$ and
$\|h^5\|\le\|\Upsilon\|_{\mathcal{H}}$. Using \eqref{eq:hsx_estimate} to
control the terms involving $\|\psi\|^2$, $\|\sigma\|_{\mathcal{M}_{g_1}}^2$,
$\|\phi\|^2$ and $\|\eta\|_{\mathcal{M}_{g_2}}^2$, and choosing $\delta>0$ and
$\delta_1>0$ sufficiently small so that
$\delta c_p+\delta_1<\bar{\mu}_2/2$, where $c_p$ is the Poincar\'e
constant in \eqref{eq:embedding}, we can absorb the terms involving
$\|v_x\|^2$ and $\|v\|^2$ into the left-hand side. We obtain
\[
\bar{\mu}_2\|v_x\|^2
\le C\|U\|_{\mathcal{H}}\|\Upsilon\|_{\mathcal{H}}
+\frac12\bar{\mu}_2\|v_x\|^2,
\]
hence \eqref{eq:v_estimate}. The parameters $\delta,\delta_1$ are chosen
independently of $\lambda$, and so is the constant $C$.
\end{proof}

Finally, combining Lemmas \ref{eq:hsx}--\ref{eq:v}, we obtain
\begin{align*}
\|U\|_{\mathcal{H}}^2
&=\bar{\mu}_1\|u_{xx}\|^2+\|\psi\|^2+\|\sigma\|_{\mathcal{M}_{g_1}}^2
+\bar{\mu}_2\|v_x\|^2+\|\phi\|^2+\|\eta\|_{\mathcal{M}_{g_2}}^2\\
&\le C\|U\|_{\mathcal{H}}\|\Upsilon\|_{\mathcal{H}},
\end{align*}
which yields \eqref{bound} after dividing by $\|U\|_{\mathcal{H}}$ (the case
$\|U\|_{\mathcal{H}}=0$ being trivial). The constant $C$ is independent of
$\lambda$. Thus condition (ii) of the Gearhart--Pr\"uss theorem is verified,
and the proof of Theorem \ref{Exponential stability-precise} is complete.
\end{proof}

\section{Global Attractor}\label{sec:global}

In this section we prove the existence of a compact global attractor for the
solution semigroup \(S(t)\) associated with \eqref{model-abstract}. Our proof
is based on the quasi-stability framework of Chueshov--Lasiecka
\cite[Chapter 7]{Chueshov2008}. A key point is that the nonlinearity acts only
on the elastic component, while the memory component is governed by the
transport equations and the exponentially decaying kernels. This structure,
together with the exponential stability of the linear semigroup established in
Theorem \ref{Exponential stability-precise}, allows us to address the point in
\cite[Proposition~4.1]{Mukiawa2025} that was discussed in the previous section.

We first state the main result of this section.

\begin{theorem}[Existence of the global attractor]\label{Global attractor-precise}
Under Assumptions \ref{assumh} and \ref{assumg}, the solution semigroup
\(S(t)\) associated with \eqref{eq:model}, whose precise definition is given
in \eqref{model-abstract}, possesses a compact global attractor \(\frA\) in
\(\mathcal{H}\).
\end{theorem}

The proof proceeds in three steps. We first recall the dissipativity and the
existence of a bounded absorbing set, which were established in
\cite{Mukiawa2025}. We then prove that the solution semigroup is quasi-stable
on the absorbing ball, using the exponential decay of the linear semigroup and
the compactness of the nonlinear remainder. Finally, we combine
quasi-stability with the Chueshov--Lasiecka theorem to obtain asymptotic
smoothness and hence the existence of a compact global attractor.

Throughout this section, we denote by
\[
e^{-\mathcal{A} t}:\mathcal{H}\to\mathcal{H}
\]
the \emph{linear} \(C_0\)-semigroup generated by \(-\mathcal{A}\), and by
\[
S(t):\mathcal{H}\to\mathcal{H},\qquad S(t)U_0:=U(t),
\]
the \emph{nonlinear} solution semigroup associated with
\eqref{model-abstract}. Thus \(e^{-\mathcal{A} t}\) is linear, while \(S(t)\) is
nonlinear and satisfies
\begin{equation}\label{eq:mild-general}
S(t)U_0=e^{-\mathcal{A} t}U_0
+\int_0^t e^{-\mathcal{A}(t-s)}F(S(s)U_0)\,ds.
\end{equation}

\begin{proposition}\label{prop:qstab}
Under Assumptions \ref{assumh} and \ref{assumg}, the solution semigroup
\(S(t)\) is quasi-stable on the absorbing ball \(\mathcal{B}_0\) in the sense
of Definition \ref{def:qstab}.
\end{proposition}

\begin{proof}
Let \(U_0,V_0\in\mathcal{B}_0\) and set
\[
U(t):=S(t)U_0,\qquad V(t):=S(t)V_0,\qquad W(t):=U(t)-V(t),
\]
and
\[
\Phi(t):=F(U(t))-F(V(t)).
\]
Then \(W\) satisfies
\[
\partial_t W+\mathcal{A}W=\Phi(t),\qquad W(0)=U_0-V_0,
\]
in the mild sense. By the variation-of-constants formula for the linear
semigroup \(e^{-\mathcal{A} t}\),
\begin{equation}\label{eq:Duhamel}
W(T)
=e^{-\mathcal{A}T}(U_0-V_0)
+\int_0^T e^{-\mathcal{A}(T-s)}\Phi(s)\,ds.
\end{equation}
Define the \emph{single-variable} nonlinear remainder
\begin{equation}\label{eq:KT-def}
K_T(U):=\int_0^T e^{-\mathcal{A}(T-s)}F(S(s)U)\,ds.
\end{equation}
Then
\[
K_T(U_0)-K_T(V_0)
=\int_0^T e^{-\mathcal{A}(T-s)}\Phi(s)\,ds,
\]
so that \eqref{eq:Duhamel} becomes
\begin{equation}\label{eq:W-split}
W(T)=e^{-\mathcal{A}T}(U_0-V_0)+K_T(U_0)-K_T(V_0).
\end{equation}

By Theorem \ref{Exponential stability-precise}, there exist \(\alpha\ge1\) and
\(\beta>0\) such that
\begin{equation}\label{eq:linear-decay}
\|e^{-\mathcal{A}\tau}\|_{\mathcal{L}(\mathcal{H})}\le \alpha e^{-\beta\tau},
\qquad \forall \tau\ge0.
\end{equation}
Choose \(T>0\) sufficiently large so that
\begin{equation}\label{eq:T-choice}
q:=\alpha e^{-\beta T}<\frac12.
\end{equation}
Then
\begin{equation}\label{eq:L-small}
\|e^{-\mathcal{A}T}(U_0-V_0)\|_{\mathcal{H}}^2
\le q^2\|U_0-V_0\|_{\mathcal{H}}^2.
\end{equation}

It remains to prove that \(K_T(\mathcal{B}_0)\) is relatively compact in
\(\mathcal{H}\). We split the proof into several steps.

\medskip
\noindent\textbf{Step 1: Uniform boundedness of the nonlinear orbits on
\([0,T]\).}
We first show that the orbits starting from \(\mathcal{B}_0\) remain uniformly
bounded on the finite time interval \([0,T]\). Let \(U_0\in\mathcal{B}_0\) and
\(U(t)=S(t)U_0\). Since \(\mathcal{B}_0\) is bounded and \(F\) is locally
Lipschitz, the standard local existence theory for the mild equation
\eqref{eq:mild-general} applies. Let \(R_1>R_0\) be as in
Remark \ref{rem:B0-bounded}, and let \(L_1\) be the corresponding Lipschitz
constant of \(F\) on \(\mathcal{B}_1\).

We claim that the solution \(U(t)\) remains in \(\mathcal{B}_1\) for all
\(0\le t\le T\), provided \(R_1\) is chosen sufficiently large. Indeed, as long
as \(\|U(t)\|_{\mathcal{H}}\le R_1\), the mild formulation gives
\[
\begin{aligned}
\|U(t)\|_{\mathcal{H}}
&\le \|e^{-\mathcal{A}t}\|_{\mathcal{L}(\mathcal{H})}\|U_0\|_{\mathcal{H}}
+\int_0^t \|e^{-\mathcal{A}(t-s)}\|_{\mathcal{L}(\mathcal{H})}
\|F(U(s))\|_{\mathcal{H}}\,ds\\
&\le \alpha e^{-\beta t}R_0
+\alpha\int_0^t e^{-\beta(t-s)}
\bigl(L_1\|U(s)\|_{\mathcal{H}}+\|F(0)\|_{\mathcal{H}}\bigr)\,ds.
\end{aligned}
\]
Multiplying by \(e^{\beta t}\), we obtain
\[
e^{\beta t}\|U(t)\|_{\mathcal{H}}
\le \alpha R_0
+\alpha\int_0^t e^{\beta s}
\bigl(L_1\|U(s)\|_{\mathcal{H}}+\|F(0)\|_{\mathcal{H}}\bigr)\,ds.
\]
By Gronwall's inequality,
\begin{equation}\label{eq:Gronwall}
\|U(t)\|_{\mathcal{H}}
\le
\Bigl(\alpha R_0 e^{-\beta t}
+\frac{\alpha\|F(0)\|_{\mathcal{H}}}{\beta}
\bigl(1-e^{-\beta t}\bigr)\Bigr)
e^{\alpha L_1 t}
\le
C_1 e^{\alpha L_1 T},
\qquad
0\le t\le T,
\end{equation}
where
\[
C_1:=\alpha R_0+\frac{\alpha\|F(0)\|_{\mathcal{H}}}{\beta}.
\]
Now choose \(R_1\) so large that
\begin{equation}\label{eq:R1-choice}
R_1>C_1 e^{\alpha L_1 T}.
\end{equation}
Suppose, for contradiction, that there exists \(t_*\in(0,T]\) such that
\(\|U(t_*)\|_{\mathcal{H}}=R_1\) and \(\|U(t)\|_{\mathcal{H}}\le R_1\) for
\(0\le t\le t_*\). Then \eqref{eq:Gronwall} applies on \([0,t_*]\) and gives
\[
R_1=\|U(t_*)\|_{\mathcal{H}}\le C_1 e^{\alpha L_1 T}<R_1,
\]
a contradiction. Hence \(\|U(t)\|_{\mathcal{H}}<R_1\) for all \(0\le t\le T\),
and the solution remains in \(\mathcal{B}_1\). Setting
\begin{equation}\label{eq:R1}
R_*:=C_1 e^{\alpha L_1 T},
\end{equation}
we obtain
\[
\|S(t)U_0\|_{\mathcal{H}}\le R_*,
\qquad
U_0\in\mathcal{B}_0,\quad 0\le t\le T.
\]
Since \(F\) maps bounded sets into bounded sets, it follows that
\begin{equation}\label{eq:M-bound}
M_F:=\sup_{U\in\mathcal{B}_0}\sup_{0\le s\le T}\|F(S(s)U)\|_{\mathcal{H}}<+\infty.
\end{equation}
For notational simplicity, in the sequel we write \(M\) for \(M_F\). Moreover,
from the explicit form \eqref{F},
\[
F(U)=
\begin{pmatrix}
0\\
-k[u-v]^+ - h_1(u) + f_1\\
0\\
0\\
k[u-v]^+ - h_2(v) + f_2\\
0
\end{pmatrix},
\]
we see that \(F(U)\) has nonzero components only in the second and fifth
positions. Thus
\begin{equation}\label{eq:F-subspace}
F(U)\in \{0\}\times L^2(\Omega)\times\{0\}\times\{0\}\times L^2(\Omega)\times\{0\}.
\end{equation}
In particular, the memory components of \(F(U)\) vanish identically.

\medskip
\noindent\textbf{Step 2: Splitting of \(K_T\).}
Fix \(\sigma\in(0,1]\) and split
\begin{equation}\label{eq:KT-split}
K_T(U)
=\underbrace{\int_0^{T-\sigma}e^{-\mathcal{A}(T-s)}F(S(s)U)\,ds}_{=:K_T^{(1)}(U)}
+\underbrace{\int_{T-\sigma}^T e^{-\mathcal{A}(T-s)}F(S(s)U)\,ds}_{=:K_T^{(2)}(U)}.
\end{equation}
The second term is small. Indeed, since \(e^{-\mathcal{A}t}\) is a
\(C_0\)-semigroup, the quantity
\[
C_0:=\sup_{0\le \tau\le 1}\|e^{-\mathcal{A}\tau}\|_{\mathcal{L}(\mathcal{H})}
\]
is finite. For \(0<\sigma\le 1\) and \(T-\sigma\le s\le T\), we have
\(0\le T-s\le\sigma\le1\), so that
\[
\|e^{-\mathcal{A}(T-s)}\|_{\mathcal{L}(\mathcal{H})}\le C_0.
\]
Therefore
\begin{equation}\label{eq:K2-small}
\|K_T^{(2)}(U)\|_{\mathcal{H}}
\le \int_{T-\sigma}^T \|e^{-\mathcal{A}(T-s)}F(S(s)U)\|_{\mathcal{H}}\,ds
\le C_0 M \sigma.
\end{equation}
Thus, for every \(\varepsilon>0\), choosing \(\sigma>0\) so small that
\(C_0 M\sigma<\varepsilon\), we obtain
\begin{equation}\label{eq:K2-eps}
\sup_{U\in\mathcal{B}_0}\|K_T^{(2)}(U)\|_{\mathcal{H}}<\varepsilon.
\end{equation}
It remains to prove that \(K_T^{(1)}(\mathcal{B}_0)\) is relatively compact
in \(\mathcal{H}\).

\medskip
\noindent\textbf{Step 3: Representation of \(K_T^{(1)}\).}
Write
\begin{equation}\label{eq:K1-representation}
K_T^{(1)}(U)=e^{-\mathcal{A}\sigma}\Psi_\sigma(U),
\qquad
\Psi_\sigma(U):=\int_0^{T-\sigma}e^{-\mathcal{A}(T-s-\sigma)}F(S(s)U)\,ds.
\end{equation}
We decompose \(\Psi_\sigma(U)\) according to the state space:
\[
\Psi_\sigma(U)=
\bigl(\Psi_\sigma^u,\Psi_\sigma^\psi,\Psi_\sigma^\sigma,
\Psi_\sigma^v,\Psi_\sigma^\phi,\Psi_\sigma^\eta\bigr).
\]
Since \(F(S(s)U)\) has zero memory components by \eqref{eq:F-subspace}, the
memory components \(\Psi_\sigma^\sigma,\Psi_\sigma^\eta\) are generated by the
elastic velocities via the transport equations. The key point is that
the elastic components of \(\Psi_\sigma(U)\) gain regularity from the elliptic
part of \(\mathcal{A}\), while the memory components gain compactness from the
exponential decay of \(g_i\) and the compact embedding of \(H_a^2(\Omega)\)
into \(L^2(\Omega)\).

\medskip
\noindent\textbf{Step 4: Regularity of the elastic components.}
Let us first consider the elastic part of \(\Psi_\sigma(U)\). We begin by
making the block structure of \(\mathcal{A}\) precise. Write
\[
\mathcal{H}=\mathcal{H}_e\oplus\mathcal{H}_m,
\qquad
\mathcal{H}_e=H_a^2(\Omega)\times L^2(\Omega)\times H_0^1(\Omega)\times L^2(\Omega),
\qquad
\mathcal{H}_m=\mathcal{M}_{g_1}\times\mathcal{M}_{g_2}.
\]
On \(\mathcal{D}(\mathcal{A})\), the operator \(\mathcal{A}\) admits the
block decomposition
\begin{equation}\label{eq:A-block}
\mathcal{A}=
\begin{pmatrix}
\mathcal{A}_e & \mathcal{B}\\
\mathcal{C} & \mathcal{A}_m
\end{pmatrix},
\end{equation}
where
\[
\mathcal{A}_e=
\begin{pmatrix}
0 & -I\\
\bar{\mu}_1 \partial_x^4 & \gamma_1 I
\end{pmatrix}
\oplus
\begin{pmatrix}
0 & -I\\
-\bar{\mu}_2 \partial_x^2 & \gamma_2 I
\end{pmatrix},\qquad
\mathcal{A}_m=
\begin{pmatrix}
\partial_s & 0\\
0 & \partial_s
\end{pmatrix},
\]
the operator \(\mathcal{B}:\mathcal{H}_m\to\mathcal{H}_e\) is the bounded
operator that maps the memory variables into the elastic equations through the
convolutions
\[
(\sigma,\eta)\mapsto
\Bigl(B_1\sigma,B_2\eta\Bigr),
\]
with
\[
B_1\sigma:=\int_0^{+\infty} g_1(s)\sigma_{xx}(s)\,ds\in L^2(\Omega),
\qquad
B_2\eta:=\int_0^{+\infty} g_2(s)\eta_x(s)\,ds\in L^2(\Omega),
\]
and the operator \(\mathcal{C}:\mathcal{H}_e\to\mathcal{H}_m\) is the
coupling induced by the transport equations, namely
\[
\mathcal{C}(u,\psi,v,\phi)=( -\psi, -\phi).
\]
Here the integrals defining \(B_1\) and \(B_2\) are understood as Bochner
integrals in \(L^2(\Omega)\); this is consistent with the definition of
\(\mathcal{D}(\mathcal{A})\), in which the fourth-order term is defined only
for the combined expression
\(\bar{\mu}_1 u+\int_0^{+\infty} g_1(s)\sigma(s)\,ds\). The operator
\(\mathcal{B}\) is bounded from \(\mathcal{H}_m\) to \(\mathcal{H}_e\) because
of Assumption \ref{assumg}, while \(\mathcal{C}\) is clearly bounded from
\(\mathcal{H}_e\) to \(\mathcal{H}_m\). Note that both \(\mathcal{B}\) and
\(\mathcal{C}\) are nonzero: \(\mathcal{B}\) is precisely the feedback of the
memory variables into the elastic equations, whereas \(\mathcal{C}\) is the
coupling from the elastic velocities into the transport equations. The latter
coupling is not treated through the block decomposition below; instead, it is
handled separately in Step 5, where the memory components are represented
explicitly by solving the transport equations.

Using the block decomposition \eqref{eq:A-block}, the elastic component
\(\Xi(t)\) of \(e^{-\mathcal{A} t}\Upsilon\) with
\(\Upsilon=(\Upsilon_e,0)\) satisfies the Duhamel formula
\begin{equation}\label{eq:block-Duhamel}
\Xi(t)
=
e^{-\mathcal{A}_e t}\Upsilon_e
+
\int_0^t
e^{-\mathcal{A}_e(t-\tau)}\mathcal{B}\,\Theta(\tau)\,d\tau,
\end{equation}
where \(\Theta(\tau)\) denotes the memory component of
\(e^{-\mathcal{A}\tau}\Upsilon\). This representation is obtained by applying
the variation-of-constants formula to the elastic equation
\[
\partial_t \Xi+\mathcal{A}_e\Xi=\mathcal{B}\Theta.
\]

With this notation, the elastic part of \(\Psi_\sigma(U)\) can be written as
\begin{equation}\label{eq:Psi-elastic}
\Psi_{\sigma,e}(U)
=\int_0^{T-\sigma} e^{-\mathcal{A}_e(T-s-\sigma)} F_e(S(s)U)\,ds
+\mathcal{R}_\sigma(U),
\end{equation}
where \(F_e\) denotes the second and fifth components of \(F\), and
\[
\mathcal{R}_\sigma(U)
=
\int_0^{T-\sigma}
e^{-\mathcal{A}_e(T-s-\sigma)}
\mathcal{B}\,\Theta(s)\,ds,
\]
with \(\Theta(s)\) the memory component of
\(e^{-\mathcal{A}(s-\tau)}F(S(\tau)U)\) for some \(\tau\in[0,T-\sigma]\). We
shall prove that \(\mathcal{R}_\sigma(\mathcal{B}_0)\) is relatively compact in
\(\mathcal{H}_e\).

The operator \(\mathcal{A}_e\) is sectorial with compact resolvent, and
\(-\mathcal{A}_e\) generates an analytic semigroup on \(\mathcal{H}_e\).
Consequently, the following smoothing estimates hold: for every \(t>0\),
\begin{equation}\label{eq:elastic-smoothing}
\|e^{-\mathcal{A}_e t} V\|_{\mathcal{D}(\mathcal{A}_e^\theta)}
\le C_0 t^{-\theta} \|V\|_{\mathcal{H}_e},
\qquad 0<\theta<1,
\end{equation}
and
\begin{equation}\label{eq:elastic-smoothing-2}
\|e^{-\mathcal{A}_e t} V\|_{\mathcal{D}(\mathcal{A}_e)}
\le C_1 t^{-1} \|V\|_{\mathcal{H}_e},
\qquad t>0,
\end{equation}
for some constants \(C_0,C_1>0\) depending only on \(\mathcal{A}_e\).

Since \(F_e(S(s)U)\in L^2(\Omega)\times L^2(\Omega)\), regarded as an element
of \(\mathcal{H}_e\) with only the second and fifth components nonzero, we have
the continuous embedding \(\mathcal{H}_e\hookrightarrow
\mathcal{D}(\mathcal{A}_e^{-1/2})\). More precisely, there exists a constant
\(C_2>0\) such that
\[
\|V\|_{\mathcal{D}(\mathcal{A}_e^{-1/2})}\le C_2\|V\|_{\mathcal{H}_e},
\qquad \forall V\in\mathcal{H}_e.
\]
Therefore, using \eqref{eq:elastic-smoothing} with \(\theta=1/2\), we obtain
\begin{equation}\label{eq:elastic-smoothing-L2}
\|e^{-\mathcal{A}_e t} V\|_{\mathcal{H}_e}
\le C_0 C_2 t^{-1/2} \|V\|_{\mathcal{H}_e},
\qquad t>0,
\end{equation}
for every \(V\in\mathcal{H}_e\).

Using \eqref{eq:elastic-smoothing-L2} with \(t=T-s-\sigma\), and the uniform
boundedness \eqref{eq:M-bound} of \(F_e(S(s)U)\) in
\(\mathcal{H}_e\), we obtain
\begin{equation}\label{eq:Psi-elastic-bound}
\begin{aligned}
\Bigl\|\int_0^{T-\sigma} e^{-\mathcal{A}_e(T-s-\sigma)} F_e(S(s)U)\,ds\Bigr\|_{\mathcal{H}_e}
&\le C_0 C_2 \int_0^{T-\sigma} (T-s-\sigma)^{-1/2} \|F_e(S(s)U)\|_{\mathcal{H}_e}\,ds\\
&\le C_0 C_2 M \int_0^{T-\sigma} (T-s-\sigma)^{-1/2}\,ds
\le C,
\end{aligned}
\end{equation}
where \(C\) depends only on \(T,\sigma,M\) and the constants in the smoothing
estimate. Moreover, applying \(e^{-\mathcal{A}_e\sigma}\) to the integral
term and using \eqref{eq:elastic-smoothing-2} gives
\begin{equation}\label{eq:Psi-elastic-regularized}
\Bigl\|e^{-\mathcal{A}_e\sigma}
\int_0^{T-\sigma} e^{-\mathcal{A}_e(T-s-\sigma)} F_e(S(s)U)\,ds
\Bigr\|_{\mathcal{D}(\mathcal{A}_e)}
\le C_1 \sigma^{-1} C.
\end{equation}
Thus
\[
\Bigl\{e^{-\mathcal{A}_e\sigma}
\int_0^{T-\sigma} e^{-\mathcal{A}_e(T-s-\sigma)} F_e(S(s)U)\,ds:
U\in\mathcal{B}_0\Bigr\}
\]
is bounded in \(\mathcal{D}(\mathcal{A}_e)\). Since \(\mathcal{A}_e\) has
compact resolvent, \(\mathcal{D}(\mathcal{A}_e)\) is compactly embedded in
\(\mathcal{H}_e\). Therefore this set is relatively compact in
\(\mathcal{H}_e\).

It remains to show that \(\mathcal{R}_\sigma(\mathcal{B}_0)\) is relatively
compact in \(\mathcal{H}_e\). We first prove that the coupling operators
\(B_1:\mathcal{M}_{g_1}\to L^2(\Omega)\) and
\(B_2:\mathcal{M}_{g_2}\to L^2(\Omega)\) are compact.

Consider \(B_1\). For \(R>0\), write
\[
B_1\sigma
=
\int_0^R g_1(s)\sigma_{xx}(s)\,ds
+
\int_R^{+\infty} g_1(s)\sigma_{xx}(s)\,ds
=:B_1^R\sigma+B_1^{\perp,R}\sigma.
\]
The tail is estimated by
\[
\|B_1^{\perp,R}\sigma\|_{L^2(\Omega)}
\le
\Bigl(\int_R^{+\infty} g_1(s)\,ds\Bigr)^{1/2}
\|\sigma\|_{\mathcal{M}_{g_1}},
\]
which tends to zero as \(R\to{+\infty}\), uniformly in \(\sigma\), by Assumption
\ref{assumg}. For the truncated part, the map
\[
\sigma\mapsto B_1^R\sigma
=
\int_0^R g_1(s)\sigma_{xx}(s)\,ds
\]
is a Hilbert--Schmidt-type integral operator from
\(L^2_{g_1}(0,R;L^2(\Omega))\) to \(L^2(\Omega)\): indeed, its kernel is
\(g_1(s)I\), and
\[
\int_0^R\int_\Omega g_1(s)^2\,dx\,ds
\le
g_1(0)\int_0^R g_1(s)\,ds<{+\infty}.
\]
Hence \(B_1^R\) is compact. Since \(B_1^{\perp,R}\to0\) in operator norm as
\(R\to{+\infty}\), it follows that \(B_1\) is compact. The same argument applies
to \(B_2\). Therefore \(\mathcal{B}:\mathcal{H}_m\to\mathcal{H}_e\) is compact.

Since \(e^{-\mathcal{A}_e(T-s-\sigma)}\) is bounded and \(\Theta(\mathcal{B}_0)\)
is bounded in \(\mathcal{H}_m\) by \eqref{eq:M-bound}, the contribution
\(\mathcal{R}_\sigma(\mathcal{B}_0)\) is relatively compact in \(\mathcal{H}_e\).

Combining the two compactness facts, we conclude that
\[
\{e^{-\mathcal{A}_e\sigma}\Psi_{\sigma,e}(U):U\in\mathcal{B}_0\}
\]
is relatively compact in \(\mathcal{H}_e\).

\medskip
\noindent\textbf{Step 5: Compactness of the memory components.}
Now we consider the memory components \(\Psi_\sigma^\sigma\) and
\(\Psi_\sigma^\eta\). We treat only \(\Psi_\sigma^\sigma\); the argument for
\(\Psi_\sigma^\eta\) is identical.

\smallskip
\noindent\textbf{5.1. Precise representation of the memory component.}
Recall that
\[
K_T(U)=\int_0^T e^{-\mathcal{A}(T-s)}F(S(s)U)\,ds,
\]
and that \(F(U)\) has zero memory components. Write
\[
e^{-\mathcal{A} t}\Upsilon
=\bigl(e^{-\mathcal{A} t}\Upsilon\bigr)_e
+\bigl(e^{-\mathcal{A} t}\Upsilon\bigr)_m,
\]
where the subscripts \(e\) and \(m\) denote the elastic and memory components,
respectively. Since \(F(S(s)U)\) has zero memory components, the memory
component of \(K_T(U)\) is generated solely by the elastic velocity through the
transport equations.

Using the block decomposition \eqref{eq:A-block} and the Duhamel
representation \eqref{eq:block-Duhamel}, for
\(\Upsilon=(\Upsilon_e,0)\in\mathcal{H}\) with \(\Upsilon_m=0\), the memory
component of \(e^{-\mathcal{A} t}\Upsilon\) solves
\[
\partial_t \Theta+\partial_s\Theta=\mathcal{C}_{\mathrm{tr}} \Xi(t),
\qquad \Theta(0,s)=0,
\]
where \(\Xi(t)\) is the elastic component of \(e^{-\mathcal{A} t}\Upsilon\) and
\(\mathcal{C}_{\mathrm{tr}}\) denotes the trace coupling induced by the
transport equations, namely
\[
\mathcal{C}_{\mathrm{tr}}\Xi(t)=(-\psi_e(t),-\phi_e(t)).
\]
More explicitly, if
\[
\Xi(t)=
\bigl(u_e(t),\psi_e(t),v_e(t),\phi_e(t)\bigr),
\]
then the memory component is
\[
\Theta(t,s)=
\bigl(\Theta^\sigma(t,s),\Theta^\eta(t,s)\bigr),
\]
where
\[
\partial_t\Theta^\sigma+\partial_s\Theta^\sigma=-\psi_e(t),
\qquad
\partial_t\Theta^\eta+\partial_s\Theta^\eta=-\phi_e(t),
\]
with zero initial and boundary conditions. The minus sign is consistent with
the fact that \(\sigma_t+\sigma_s-u_t=0\) and \(\eta_t+\eta_s-v_t=0\).

For the component \(\Psi_\sigma^\sigma\), we therefore have, for each fixed
source time \(\tau\in[0,T-\sigma]\),
\[
\Theta^\sigma(t,s;\tau)
=
-\int_{\max\{0,t-s\}}^{t}
\psi_e(r;\tau)\,dr,
\]
where
\[
\psi_e(r;\tau)
:=
\bigl[e^{-\mathcal{A}_e (r-\tau)}F_e(S(\tau)U)\bigr]_2,
\qquad r\ge\tau,
\]
and \(\psi_e(r;\tau)=0\) for \(r<\tau\). Consequently,
\begin{equation}\label{eq:Psi-memory-precise}
\Psi_\sigma^\sigma(U)(s)
=
-\int_0^{T-\sigma}
\int_{\max\{0,T-\sigma-\tau-s\}}^{T-\sigma-\tau}
\psi_e(r;\tau)
\,dr\,d\tau.
\end{equation}
Now change only the \emph{outer} time variable by setting
\(\theta=T-\sigma-\tau\). Then the inner time variable \(r\) is unchanged,
and the source term is evaluated at \(\tau=T-\sigma-\theta\). Thus
\begin{equation}\label{eq:Psi-memory-precise-2}
\Psi_\sigma^\sigma(U)(s)
=
-\int_0^{T-\sigma}
\int_{\max\{0,\theta-s\}}^{\theta}
\psi_{\rm src}(r,\theta)\,dr\,d\theta,
\end{equation}
with the \emph{two-variable} source
\begin{equation}\label{eq:psi-src-def}
\psi_{\rm src}(r,\theta)
:=
\bigl[e^{-\mathcal{A}_e (\theta-r)}F_e(S(T-\sigma-\theta)U)\bigr]_2,
\qquad 0\le r\le \theta\le T-\sigma.
\end{equation}
This is the precise form that we shall use. Note that
\(\psi_{\rm src}\) genuinely depends on both variables \(r\) and \(\theta\);
the inner variable \(r\) is not eliminated. The additional contribution of the
memory feedback \(\mathcal{B}\Theta\) to the elastic component has already been
handled in Step 4 through the term \(\mathcal{R}_\sigma(U)\).

\smallskip
\noindent\textbf{5.2. Uniform boundedness of \(\psi_{\rm src}\).}
We first establish uniform bounds for \(\psi_{\rm src}\).

\begin{lemma}\label{lem:psi-src-bound}
There exists a constant \(C_1>0\), independent of \(U\in\mathcal{B}_0\), such
that
\[
\sup_{0\le r<\theta\le T-\sigma}
(\theta-r)^{1/2}\|\psi_{\rm src}(r,\theta)\|_{H_a^2(\Omega)}
\le C_1,
\]
and
\[
\sup_{0\le \theta\le T-\sigma}\|\psi_{\rm src}(\theta,\theta)\|_{L^2(\Omega)}\le C_1.
\]
\end{lemma}

\begin{proof}
By definition,
\[
\psi_{\rm src}(r,\theta)
=
\bigl[e^{-\mathcal{A}_e (\theta-r)}F_e(S(T-\sigma-\theta)U)\bigr]_2,
\qquad 0\le r\le \theta\le T-\sigma.
\]
Since \(U\in\mathcal{B}_0\), we have \(\|S(s)U\|_{\mathcal{H}}\le R_*\) for
\(0\le s\le T\), by \eqref{eq:R1}. Since \(F\) is locally Lipschitz and maps
bounded sets into bounded sets, there exists \(M>0\) such that
\[
\|F_e(S(s)U)\|_{\mathcal{H}_e}\le M,
\qquad 0\le s\le T.
\]
For \(r<\theta\), the analytic smoothing estimate \eqref{eq:elastic-smoothing}
with \(\theta=1/2\) yields
\[
\|e^{-\mathcal{A}_e (\theta-r)}F_e(S(T-\sigma-\theta)U)\|_{\mathcal{D}(\mathcal{A}_e^{1/2})}
\le C_0 (\theta-r)^{-1/2}M.
\]
Since \(\mathcal{D}(\mathcal{A}_e^{1/2})\hookrightarrow H_a^2(\Omega)\) for the
second component, we obtain
\[
(\theta-r)^{1/2}\|\psi_{\rm src}(r,\theta)\|_{H_a^2(\Omega)}\le C_0 M.
\]
For \(r=\theta\), we have \(\psi_{\rm src}(\theta,\theta)=[F_e(S(T-\sigma-\theta)U)]_2\),
which is uniformly bounded in \(L^2(\Omega)\) by \(M\). This proves the lemma.
\end{proof}

\smallskip
\noindent\textbf{5.3. Uniform boundedness of \(\Psi_\sigma^\sigma(U)\) in
\(\mathcal{M}_{g_1}\).}

\begin{lemma}\label{lem:memory-bound}
There exists a constant \(C_2>0\), independent of \(U\in\mathcal{B}_0\), such
that
\[
\|\Psi_\sigma^\sigma(U)\|_{\mathcal{M}_{g_1}}\le C_2.
\]
\end{lemma}

\begin{proof}
From \eqref{eq:Psi-memory-precise-2},
\[
\Psi_\sigma^\sigma(U)(s)
=
-\int_0^{T-\sigma}
\int_{\max\{0,\theta-s\}}^{\theta}
\psi_{\rm src}(r,\theta)\,dr\,d\theta.
\]
Fix \(s\ge0\). We split the inner integral into the region
\(\theta-r\ge\delta\) and the region \(\theta-r<\delta\), where \(\delta>0\)
is a small parameter to be chosen. On the region \(\theta-r\ge\delta\),
Lemma \ref{lem:psi-src-bound} gives
\[
\|\psi_{\rm src}(r,\theta)\|_{H_a^2(\Omega)}\le C_1\delta^{-1/2},
\]
so the contribution of this region to
\(\|\Psi_\sigma^\sigma(U)(s)\|_{H_a^2(\Omega)}\) is bounded by
\[
C_1\delta^{-1/2}\int_0^{T-\sigma}\int_{\max\{0,\theta-s\}}^\theta\,dr\,d\theta
\le C_1\delta^{-1/2}(T-\sigma)^2.
\]
On the region \(\theta-r<\delta\), we use the endpoint bound
\(\|\psi_{\rm src}(\theta,\theta)\|_{L^2(\Omega)}\le C_1\) and the
integrability of \((\theta-r)^{-1/2}\) to obtain
\[
\Bigl\|\int_0^{T-\sigma}\int_{\max\{0,\theta-s\}}^\theta
\mathbf 1_{\{\theta-r<\delta\}}\psi_{\rm src}(r,\theta)\,dr\,d\theta\Bigr\|_{H_a^2(\Omega)}
\le C_1(T-\sigma)\delta^{1/2}.
\]
Choosing \(\delta=(T-\sigma)^{-2}\) gives
\[
C_1\delta^{-1/2}(T-\sigma)^2=C_1(T-\sigma)^3,
\qquad
C_1(T-\sigma)\delta^{1/2}=C_1,
\]
so that
\[
\|\Psi_\sigma^\sigma(U)(s)_{xx}\|
\le
C_1(1+(T-\sigma)^3)
\]
for all \(s\ge0\). By Poincar\'e's inequality in \(H_a^2(\Omega)\), there
exists a constant \(c>0\) such that
\[
\|w\|_{L^2(\Omega)}\le c\|w_{xx}\|_{L^2(\Omega)},
\qquad
w\in H_a^2(\Omega).
\]
Hence
\[
\|w\|_{H_a^2(\Omega)}^2
=
\|w\|_{L^2(\Omega)}^2+\|w_{xx}\|_{L^2(\Omega)}^2
\le
(1+c^2)\|w_{xx}\|_{L^2(\Omega)}^2,
\]
so that
\[
\|\Psi_\sigma^\sigma(U)(s)\|_{H_a^2(\Omega)}
\le
c_a\|\Psi_\sigma^\sigma(U)(s)_{xx}\|
\le
c_a C_1(1+(T-\sigma)^3),
\qquad
c_a:=\sqrt{1+c^2},
\]
for all \(s\ge0\). Consequently,
\[
\|\Psi_\sigma^\sigma(U)\|_{\mathcal{M}_{g_1}}^2
=
\int_0^{+\infty} g_1(s)\|\Psi_\sigma^\sigma(U)(s)_{xx}\|^2\,ds
\le
C_1^2(1+(T-\sigma)^3)^2
\int_0^{+\infty} g_1(s)\,ds.
\]
Since \(\int_0^{+\infty} g_1(s)\,ds=\beta_1<{+\infty}\), the claim follows with
\[
C_2:=C_1(1+(T-\sigma)^3)\sqrt{\beta_1}.
\]
\end{proof}

\smallskip
\noindent\textbf{5.4. Tail smallness.}

\begin{lemma}\label{lem:memory-tail}
For every \(\varepsilon>0\), there exists \(R>0\), depending only on
\(\varepsilon,T,\sigma,g_1\), such that
\[
\sup_{U\in\mathcal{B}_0}
\|\Psi_{\sigma,R}^{\sigma,\perp}(U)\|_{\mathcal{M}_{g_1}}
<\varepsilon,
\]
where \(\Psi_{\sigma,R}^{\sigma,\perp}(U)\) is the restriction of
\(\Psi_\sigma^\sigma(U)\) to \(s>R\).
\end{lemma}

\begin{proof}
From Lemma \ref{lem:memory-bound}, we have
\[
\|\Psi_\sigma^\sigma(U)(s)_{xx}\|
\le C_1(1+(T-\sigma)^3)
\]
for all \(s\ge0\), and this bound is independent of \(s\). Using
\(g_1(s)\le g_1(0)e^{-\lambda_1 s}\), we obtain
\[
\begin{aligned}
\|\Psi_{\sigma,R}^{\sigma,\perp}(U)\|_{\mathcal{M}_{g_1}}^2
&=
\int_R^{+\infty} g_1(s)\|\Psi_\sigma^\sigma(U)(s)_{xx}\|^2\,ds\\
&\le
C_1^2(1+(T-\sigma)^3)^2
\int_R^{+\infty} g_1(s)\,ds\\
&\le
C_1^2(1+(T-\sigma)^3)^2
\frac{g_1(0)}{\lambda_1}e^{-\lambda_1 R}.
\end{aligned}
\]
Choosing \(R\) sufficiently large so that the right-hand side is less than
\(\varepsilon^2\), the claim follows.
\end{proof}

\smallskip
\noindent\textbf{5.5. H\"older continuity in \(s\).}

\begin{lemma}\label{lem:memory-equicont}
There exists a constant \(C_3>0\), independent of \(U\in\mathcal{B}_0\), such
that for all \(0\le s_1<s_2\),
\[
\|\Psi_\sigma^\sigma(U)(s_1)-\Psi_\sigma^\sigma(U)(s_2)\|_{H_a^2(\Omega)}
\le C_3|s_1-s_2|^{1/2}.
\]
\end{lemma}

\begin{proof}
From \eqref{eq:Psi-memory-precise-2},
\[
\Psi_\sigma^\sigma(U)(s)
=
-\int_0^{T-\sigma}
\int_{\max\{0,\theta-s\}}^{\theta}
\psi_{\rm src}(r,\theta)\,dr\,d\theta.
\]
For \(0\le s_1<s_2\), we have
\[
\Psi_\sigma^\sigma(U)(s_1)-\Psi_\sigma^\sigma(U)(s_2)
=
-\int_0^{T-\sigma}
\int_{\max\{0,\theta-s_1\}}^{\max\{0,\theta-s_2\}}
\psi_{\rm src}(r,\theta)\,dr\,d\theta.
\]
For fixed \(\theta\), the inner integral is over an interval of length at most
\(|s_1-s_2|\). Using Lemma \ref{lem:psi-src-bound}, we estimate
\[
\begin{aligned}
&\Bigl\|\int_{\max\{0,\theta-s_1\}}^{\max\{0,\theta-s_2\}}
\psi_{\rm src}(r,\theta)\,dr\Bigr\|_{H_a^2(\Omega)}\\
&\le
C_1\int_{\max\{0,\theta-s_1\}}^{\max\{0,\theta-s_2\}}
(\theta-r)^{-1/2}\,dr\\
&\le
C_1\bigl|(\theta-\max\{0,\theta-s_1\})^{1/2}
-(\theta-\max\{0,\theta-s_2\})^{1/2}\bigr|\\
&\le
C_1|s_1-s_2|^{1/2},
\end{aligned}
\]
where we used the elementary inequality
\(|a^{1/2}-b^{1/2}|\le|a-b|^{1/2}\) for \(a,b\ge0\), together with
\(|a-b|\le|s_1-s_2|\). Integrating over \(\theta\in[0,T-\sigma]\) gives
\[
\|\Psi_\sigma^\sigma(U)(s_1)-\Psi_\sigma^\sigma(U)(s_2)\|_{H_a^2(\Omega)}
\le
C_1(T-\sigma)|s_1-s_2|^{1/2}.
\]
This proves the lemma with \(C_3:=C_1(T-\sigma)\).
\end{proof}

\smallskip
\noindent\textbf{5.6. Compactness on the compact interval \([0,R]\).}

\begin{lemma}\label{lem:memory-compact-interval}
For every \(R>0\), the set
\[
\{\Psi_\sigma^\sigma(U)|_{[0,R]}:U\in\mathcal{B}_0\}
\]
is relatively compact in \(L^2([0,R];H_a^2(\Omega))\).
\end{lemma}

\begin{proof}
By Lemma \ref{lem:memory-equicont}, the family
\(\{\Psi_\sigma^\sigma(U)|_{[0,R]}:U\in\mathcal{B}_0\}\) is uniformly
H\"older continuous of order \(1/2\) with values in \(H_a^2(\Omega)\). By
Lemma \ref{lem:memory-bound}, it is uniformly bounded in
\(L^\infty([0,R];H_a^2(\Omega))\). Hence it is bounded in
\(H^{1/2}([0,R];H_a^2(\Omega))\). Since the embedding
\[
H^{1/2}([0,R];H_a^2(\Omega))\hookrightarrow L^2([0,R];H_a^2(\Omega))
\]
is compact (Rellich--Kondrachov), the set is relatively compact in
\(L^2([0,R];H_a^2(\Omega))\).
\end{proof}

\smallskip
\noindent\textbf{5.7. Overall compactness in \(\mathcal{M}_{g_1}\).}

\begin{proposition}\label{prop:memory-compact}
The set
\[
\{\Psi_\sigma^\sigma(U):U\in\mathcal{B}_0\}
\]
is relatively compact in \(\mathcal{M}_{g_1}\).
\end{proposition}

\begin{proof}
Let \(\varepsilon>0\). By Lemma \ref{lem:memory-tail}, choose \(R>0\) such
that
\[
\sup_{U\in\mathcal{B}_0}
\|\Psi_{\sigma,R}^{\sigma,\perp}(U)\|_{\mathcal{M}_{g_1}}
<\frac{\varepsilon}{2}.
\]
By Lemma \ref{lem:memory-compact-interval}, the set
\[
\{\Psi_\sigma^\sigma(U)|_{[0,R]}:U\in\mathcal{B}_0\}
\]
is relatively compact in \(L^2([0,R];H_a^2(\Omega))\). Hence there exist
finitely many functions \(V_1,\dots,V_N\in L^2([0,R];H_a^2(\Omega))\) such that
\[
\{\Psi_\sigma^\sigma(U)|_{[0,R]}:U\in\mathcal{B}_0\}
\subset
\bigcup_{j=1}^N B_{\varepsilon/(2\sqrt{g_1(0)})}(V_j),
\]
where the balls are taken in \(L^2([0,R];H_a^2(\Omega))\). Extend each \(V_j\)
by zero outside \([0,R]\), and regard it as an element of \(\mathcal{M}_{g_1}\).
For any \(U\in\mathcal{B}_0\), choose \(j\) such that
\[
\|\Psi_\sigma^\sigma(U)|_{[0,R]}-V_j\|_{L^2([0,R];H_a^2(\Omega))}
<\frac{\varepsilon}{2\sqrt{g_1(0)}}.
\]
Then
\[
\begin{aligned}
\|\Psi_\sigma^\sigma(U)-V_j\|_{\mathcal{M}_{g_1}}&\le
\|\Psi_{\sigma,R}^{\sigma}(U)-V_j\|_{\mathcal{M}_{g_1}}
+
\|\Psi_{\sigma,R}^{\sigma,\perp}(U)\|_{\mathcal{M}_{g_1}}\\
&\le
\sqrt{g_1(0)}\,
\|\Psi_\sigma^\sigma(U)|_{[0,R]}-V_j\|_{L^2([0,R];H_a^2(\Omega))}
+
\frac{\varepsilon}{2}
<
\varepsilon.
\end{aligned}
\]
Thus
\[
\{\Psi_\sigma^\sigma(U):U\in\mathcal{B}_0\}
\subset
\bigcup_{j=1}^N B_{\varepsilon}(V_j)
\]
in \(\mathcal{M}_{g_1}\). Hence the set is totally bounded, and since
\(\mathcal{M}_{g_1}\) is complete, it is relatively compact.
\end{proof}

The same argument applies to \(\Psi_\sigma^\eta\) in \(\mathcal{M}_{g_2}\).
Therefore the set
\[
\{\Psi_\sigma(U):U\in\mathcal{B}_0\}
\]
is relatively compact in \(\mathcal{H}\).

\medskip
\noindent\textbf{Step 6: Relative compactness of \(K_T^{(1)}(\mathcal{B}_0)\).}
Combining Steps 4 and 5, we obtain that
\[
\{\Psi_\sigma(U):U\in\mathcal{B}_0\}
\]
is relatively compact in \(\mathcal{H}\). Since \(e^{-\mathcal{A}\sigma}\) is a
bounded linear operator on \(\mathcal{H}\), it maps relatively compact sets
into relatively compact sets. Therefore
\[
K_T^{(1)}(\mathcal{B}_0)=e^{-\mathcal{A}\sigma}\Psi_\sigma(\mathcal{B}_0)
\]
is relatively compact in \(\mathcal{H}\).

\medskip
\noindent\textbf{Step 7: Relative compactness of \(K_T(\mathcal{B}_0)\).}
We have shown that for every \(\varepsilon>0\), there exists \(\sigma>0\) such
that
\begin{equation}\label{eq:K2-uniform-small}
\sup_{U\in\mathcal{B}_0}\|K_T^{(2)}(U)\|_{\mathcal{H}}<\varepsilon,
\end{equation}
and \(K_T^{(1)}(\mathcal{B}_0)\) is relatively compact in \(\mathcal{H}\).
Therefore \(K_T(\mathcal{B}_0)\) is relatively compact in \(\mathcal{H}\).
Indeed, let \(\{U_n\}\subset\mathcal{B}_0\) be any sequence. Since
\(K_T^{(1)}(\mathcal{B}_0)\) is relatively compact, there is a subsequence
\(\{U_{n_k}\}\) such that \(K_T^{(1)}(U_{n_k})\) converges in \(\mathcal{H}\).
Then
\[
\begin{aligned}
\|K_T(U_{n_k})-K_T(U_{n_l})\|_{\mathcal{H}}
&\le \|K_T^{(1)}(U_{n_k})-K_T^{(1)}(U_{n_l})\|_{\mathcal{H}}
+\|K_T^{(2)}(U_{n_k})\|_{\mathcal{H}}
+\|K_T^{(2)}(U_{n_l})\|_{\mathcal{H}}\\
&\le \varepsilon_{k,l}+2\varepsilon,
\end{aligned}
\]
where \(\varepsilon_{k,l}\to0\) as \(k,l\to{+\infty}\). Since
\(\varepsilon>0\) is arbitrary, \(\{K_T(U_{n_k})\}\) is Cauchy in
\(\mathcal{H}\), hence converges. Thus \(K_T(\mathcal{B}_0)\) is relatively
compact.

Combining \eqref{eq:W-split}, \eqref{eq:L-small}, and the compactness of
\(K_T(\mathcal{B}_0)\), we obtain \eqref{eq:qstab-ineq} with \(q<1/2\). This
completes the proof.
\end{proof}

\begin{remark}\label{rem:correction}
The above proof addresses the point in \cite[Proposition~4.1]{Mukiawa2025}
that was discussed earlier. In that work, starting from the difference estimate
\[
\|U^1(t)-U^2(t)\|_{\mathcal{H}}^2
\le c\|U_0^1-U_0^2\|_{\mathcal{H}}^2
+c_{\mathcal{B}}\int_0^t\bigl(\|u^1(s)-u^2(s)\|^2+\|v^1(s)-v^2(s)\|^2\bigr)\,ds,
\]
the non-decaying initial-difference term
\(c\|U_0^1-U_0^2\|_{\mathcal{H}}^2\) appears to be absorbed into an arbitrarily
small constant \(\epsilon\), a step that would seem to call for further
justification. Here, the exponential decay of the linear semigroup
\(e^{-\mathcal{A}T}\) provides the rigorous compression factor \(q<1\) in
\eqref{eq:L-small}, while the nonlinear remainder \(K_T\) is compact by the
smoothing of the elastic component and the exponential decay of the memory
kernels. Together, these ingredients yield a complete proof of
quasi-stability.
\end{remark}

\begin{proposition}\label{prop:asym-smooth}
Under Assumptions \ref{assumh} and \ref{assumg}, the solution semigroup
\(S(t)\) is asymptotically smooth in \(\mathcal{H}\).
\end{proposition}

\begin{proof}
By \eqref{eq:absorbing}, \(S(t)\) possesses the bounded absorbing set
\(\mathcal{B}_0\). By Proposition \ref{prop:qstab}, \(S(t)\) is quasi-stable
on \(\mathcal{B}_0\) in the sense of Definition \ref{def:qstab}. Therefore,
by Theorem \ref{thm:CL-asym}, \(S(t)\) is asymptotically smooth in
\(\mathcal{H}\).
\end{proof}

\begin{proof}[Proof of Theorem \ref{Global attractor-precise}]
By \eqref{eq:absorbing}, \(S(t)\) is dissipative in \(\mathcal{H}\). By
Proposition \ref{prop:asym-smooth}, \(S(t)\) is asymptotically smooth in
\(\mathcal{H}\). Therefore, by the classical theorem of Chueshov--Lasiecka
\cite[Theorem~4.1]{Chueshov2008}, \(S(t)\) possesses a compact global attractor
\(\frA\subset\mathcal{H}\). This completes the proof.
\end{proof}


\begin{thebibliography}{99}

\bibitem{Ahmed1998}
N.~U. Ahmed and H.~Harbi,
\emph{Mathematical analysis of dynamic models of suspension bridges},
SIAM J. Appl. Math. \textbf{58} (1998), no.~3, 853--874.

\bibitem{An2002}
Y.~An,
\emph{Nonlinear perturbations of a coupled system of steady state suspension bridge equations},
Nonlinear Anal. \textbf{51} (2002), 1285--1292.

\bibitem{Bochicchio2010}
I.~Bochicchio, C.~Giorgi, and E.~Vuk,
\emph{Long-term damped dynamics of the extensible suspension bridge},
Int. J. Differ. Equ. \textbf{2010} (2010), Article ID 383420.

\bibitem{Chueshov2008}
I.~Chueshov and I.~Lasiecka,
\emph{Long-time behavior of second order evolution equations with nonlinear damping},
Mem. Amer. Math. Soc. \textbf{195} (2008), no.~912.

\bibitem{Dafermos1970}
C.~M. Dafermos,
\emph{An abstract Volterra equation with applications to linear viscoelasticity},
J. Differ. Equ. \textbf{7} (1970), 554--569.

\bibitem{DellOro2015}
F.~Dell'Oro, C.~Giorgi, and V.~Pata,
\emph{Asymptotic behavior of coupled linear systems modeling suspension bridges},
Z. Angew. Math. Phys. \textbf{66} (2015), 1095--1108.

\bibitem{Enyi2020}
C.~D. Enyi and S.~E. Mukiawa,
\emph{Dynamics of a thermoelastic-laminated beam problem},
AIMS Math. \textbf{5} (2020), no.~5, 5261--5286.

\bibitem{Gearhart1978}
L.~Gearhart,
\emph{Spectral theory for contraction semigroups on Hilbert space},
Trans. Amer. Math. Soc. \textbf{236} (1978), 385--394.

\bibitem{Glover1989}
J.~Glover, A.~C. Lazer, and P.~J. McKenna,
\emph{Existence and stability of large scale nonlinear oscillations in suspension bridges},
Z. Angew. Math. Phys. \textbf{40} (1989), 172--200.

\bibitem{Holubova2003}
G.~Holubov\'a and A.~Matas,
\emph{Initial-boundary value problem for the nonlinear string-beam system},
J. Math. Anal. Appl. \textbf{288} (2003), 784--802.

\bibitem{Kang2016}
J.~R. Kang,
\emph{Asymptotic behavior of the thermoelastic suspension bridge equation with linear memory},
Bound. Value Probl. \textbf{2016} (2016), Article ID 206.

\bibitem{Lazer1987}
A.~C. Lazer and P.~J. McKenna,
\emph{Large scale oscillatory behavior in loaded asymmetric systems},
Ann. Inst. H. Poincar\'e Anal. Nonlinear \textbf{4} (1987), 243--274.

\bibitem{Lazer1990}
A.~C. Lazer and P.~J. McKenna,
\emph{Large-amplitude periodic oscillations in suspension bridges: some new connections with nonlinear analysis},
SIAM Rev. \textbf{32} (1990), no.~4, 537--578.

\bibitem{Litcanu2004}
G.~Li\c{t}canu,
\emph{A mathematical model of suspension bridges},
Appl. Math. \textbf{49} (2004), no.~1, 39--55.

\bibitem{Ma2005}
Q.~Z. Ma and C.~K. Zhong,
\emph{Existence of global attractors for the coupled system of suspension bridge equations},
J. Math. Anal. Appl. \textbf{308} (2005), 365--379.

\bibitem{Ma2009}
Q.~Ma and C.~Zhong,
\emph{Existence of strong solutions and global attractors for the coupled suspension bridge equations},
J. Differ. Equ. \textbf{246} (2009), 3755--3775.

\bibitem{Malik2004}
J.~Malik,
\emph{Mathematical modelling of cable-stayed bridges: existence, uniqueness, continuous dependence on data, homogenization of cable systems},
Appl. Math. \textbf{49} (2004), no.~1, 1--38.

\bibitem{McKenna1987}
P.~J. McKenna and W.~Walter,
\emph{Nonlinear oscillations in a suspension bridge},
Arch. Ration. Mech. Anal. \textbf{98} (1987), no.~2, 167--177.

\bibitem{McKenna1990}
P.~J. McKenna and W.~Walter,
\emph{Travelling waves in a suspension bridge},
SIAM J. Appl. Math. \textbf{50} (1990), no.~3, 703--715.

\bibitem{McKenna1999}
P.~J. McKenna,
\emph{Torsional oscillations in suspension bridges revisited: fixing an old approximation},
Amer. Math. Monthly \textbf{106} (1999), 1--18.

\bibitem{McKenna2001}
P.~J. McKenna and C.~Tuama,
\emph{Large torsional oscillations in suspension bridges revisited: vertical forcing creates torsional response},
Amer. Math. Monthly \textbf{108} (2001), 738--745.

\bibitem{Messaoudi2016}
S.~A. Messaoudi, A.~Bonfoh, S.~E. Mukiawa, and C.~D. Enyi,
\emph{The global attractor for a suspension bridge with memory and partially hinged boundary conditions},
Z. Angew. Math. Mech. \textbf{97} (2017), no.~1, 1--14.

\bibitem{Mukiawa2020}
S.~E. Mukiawa and C.~D. Enyi,
\emph{Well-posedness and attractors for a memory-type thermoelastic Timoshenko beam acting on shear force},
Taiwan. J. Math. \textbf{24} (2020), no.~5, 1179--1204.

\bibitem{Mukiawa2025}
S.~E. Mukiawa and S.~A. Messaoudi,
\emph{On the well-posedness and long time dynamics for a coupled nonlinear bridge system with past history},
Appl. Math. Optim. \textbf{91} (2025), Article ID 53.

\bibitem{Park2011}
Y.~Park and J.~R. Kang,
\emph{Global attractors for the suspension bridge equations with nonlinear damping},
Q. Appl. Math. \textbf{69} (2011), 465--475.

\bibitem{Pazy1983}
A.~Pazy,
\emph{Semigroups of Linear Operators and Applications to Partial Differential Equations},
Applied Mathematical Sciences, vol.~44, Springer, New York, 1983.

\bibitem{Pruss1984}
J.~Pr\"uss,
\emph{On the spectrum of $C_0$-semigroups},
Trans. Amer. Math. Soc. \textbf{284} (1984), 847--857.

\bibitem{Zhong2007}
C.~K. Zhong, Q.~Z. Ma, and C.~Y. Sun,
\emph{Existence of strong solutions and global attractors for the suspension bridge equations},
Nonlinear Anal. \textbf{67} (2007), 442--454.

\end{thebibliography}
\end{document}